\documentclass[10pt]{article}
\usepackage{amssymb}
\usepackage{tikz}
\usepackage{tikz-cd}
\usetikzlibrary{matrix, arrows,  decorations.markings,  patterns,  plotmarks}
\usepackage{amsmath}
\usepackage{amssymb}
\usepackage{amsthm}

\usepackage[backend=bibtex, style=alphabetic, maxcitenames=5, maxbibnames=9 ]{biblatex}
\renewbibmacro{in:}{} %This Modifies the In: that otherwise appears in the bib. 
\usepackage{subcaption}

\usepackage{hyperref}
\usepackage{cleveref}

\newtheorem{theorem}{Theorem}[subsection]
\newtheorem*{theorem*}{Theorem}
\newtheorem{lemma}[theorem]{Lemma}
\newtheorem*{lemma*}{Lemma}
\newtheorem{prop}[theorem]{Proposition}
\newtheorem{corollary}[theorem]{Corollary}
\newtheorem{definition}[theorem]{Definition}
\newtheorem{claim}[theorem]{Claim}
\newtheorem{example}[theorem]{Example}
\newtheorem{conjecture}[theorem]{Conjecture}
\newtheorem{question}[theorem]{Question}
\newtheorem{remark}[theorem]{Remark}
\newtheorem{assumption}[theorem]{Assumption}

\crefname{claim}{claim}{claims}
\crefname{prop}{proposition}{propositions}
\crefname{prop}{proposition}{propositions}

\newcommand{\C}{\mathbb{C}}

\newcommand{\ZZ}{\mathbb{Z}}
\newcommand{\CC}{\mathbb{C}}
\newcommand{\RR}{\mathbb{R}}

\newcommand{\into}{\hookrightarrow}
\newcommand{\tensor}{\otimes}

\newcommand{\del}{\nabla}
\newcommand{\w}{\mathfrak w}

\DeclareMathOperator{\dTrop}{dTrop}

\DeclareMathOperator{\sgn}{sgn}

\DeclareMathOperator{\aff}{aff}
\DeclareMathOperator{\pre}{pre}

\newcommand{\CP}{\mathbb{CP}}

\DeclareMathOperator{\val}{val}

\newcommand{\opm}%{\circlerighthalfblack }
{{\bullet/\circ}}
\renewcommand{\Im}{\text{Im}}

\DeclareMathOperator{\Fuk}{Fuk}
\DeclareMathOperator{\Supp}{Supp}

\DeclareMathOperator{\Coh}{Coh}
\newcommand{\CF}{{CF^\bullet}}

\graphicspath{{./figures/}}
\begin{document}
\title{Tropical Lagrangians in toric del-Pezzo surfaces}
\author{Jeffrey Hicks}
\maketitle

\begin{abstract}
	We look at how one can construct from the data of a dimer model a Lagrangian submanifold in $(\CC^*)^n$ whose valuation projection approximates a tropical hypersurface.
	Each face of the dimer corresponds to a Lagrangian disk with boundary on our tropical Lagrangian submanifold, forming a Lagrangian mutation seed. 
	Using this we find tropical Lagrangian tori $L_{T^2}$ in the complement of a smooth anticanonical divisor of a toric del-Pezzo whose wall-crossing transformations match those of monotone SYZ fibers.
	An example is worked out for the mirror pair $(\CP^2\setminus E, W), \check X_{9111}$. 
	We find a symplectomorphism of $\CP^2\setminus E$ interchanging $L_{T^2}$ and a SYZ fiber.
	Evidence is provided that this symplectomorphism is mirror to fiberwise Fourier-Mukai transform on $\check X_{9111}$. 
\end{abstract}
\setcounter{tocdepth}{2}
\tableofcontents
\section{Introduction}
\label{subsec:syz}
\subsection{Homological Mirror Symmetry and Tropical Geometry}
Tropical geometry plays an important role in mirror symmetry, a duality proposed in \cite{candelas1991pair} between symplectic geometry on a space $X$, and complex geometry on a mirror space $\check X$. 
A proposed mechanism for constructing pairs of mirror geometries comes from SYZ mirror symmetry (\cite{strominger1996mirror}) where $X$ and $\check X$ have dual Lagrangian torus fibrations over a common affine manifold $Q$. 
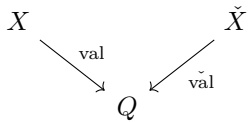
\[
    \begin{tikzcd}
        X \arrow{dr}{\val} & & \check X \arrow{dl}{\check \val}\\
         & Q
    \end{tikzcd}
\]
From this viewpoint, mirror symmetry is recovered by degenerating the symplectic geometry of $X$ and complex geometry of $\check X$ to tropical geometry on the base $Q$. 
In the complex setting, this degeneration was studied by \cite{kontsevich2001homological,mikhalkin2005enumerative}, where a correspondence between the valuations of complex curves (called the amoeba) and tropical curves was established.
More recently, tropical-Lagrangian correspondences have been constructed in the parallel works of \cite{mikhalkin2019examples,matessi2018lagrangian,hicks2020tropical,mak2019tropically}.
These papers show that for a given tropical curve $V\subset Q$ there exists a Lagrangian submanifold $L(V)\subset X$ with $\val(X)$ approximating $V$. 

A more precise relation of these two geometries is the homological mirror symmetry conjecture of \cite{kontsevich1994homological}.
This predicts that Lagrangian submanifolds of $X$ and complex submanifolds of $\check X$ should be compared as objects via a mirror functor between the categories $\Fuk(X)$ and $D^b\Coh(\check X)$. 
An expectation is that homological and SYZ mirror symmetry interact by relating Lagrangian torus fibers of $\val: X\to Q$ to skyscraper sheaves of points on $\check X$, and sections of the Lagrangian torus fibration to line bundles of $\check X$. 

This intuition was used in \cite{abouzaid2009morse} which proved that the Fukaya-Seidel category $\Fuk((\CC^*)^n, W_\Sigma)$ is equivalent to $D^b\Coh(\check X_\Sigma)$, the derived category of coherent sheaves on a mirror toric manifold.
This was achieved by using tropical geometry to construct Lagrangian sections of $\val:(\CC^*)^n\to \RR^n$, and to show that these were mirror to line bundles on $\check X_\Sigma$. 
In \cite{hicks2020tropical}, it was shown that the tropical-Lagrangian and tropical-complex correspondences are compatible with this mirror functor, in the sense that when a tropical hypersurface $V$ is approximated by $\check \val(D)$ for a divisor $D$, the Lagrangian $L(V)$ is mirror to the sheaf $\mathcal O_D$. 
This extends the relation between homological and SYZ mirror symmetry to sheaves beyond line bundles and skyscrapers of points.

\subsection{Wall-Crossings and Lagrangian Mutation}
Lagrangian submanifolds have a moduli as objects of the Fukaya category. 
For example, the moduli space of Lagrangian torus fibers of the SYZ fibration equipped with local systems is expected to be the mirror space $\check X$.
A hands-on approach to understanding this moduli space is to build complex coordinate charts. 
These coordinate functions are locally constructed using the flux homomorphism between SYZ fibers, where the flux is weighted by the local systems.
An expectation is that Gromov-Witten potential, a weighted count of holomorphic disks with boundary on these Lagrangian tori, gives a holomorphic function on the moduli space.
A first example are the product tori in $\CC^2$, which bound two holomorphic disks, and whose GW potential is given by the sum of the two standard flux coordinates.

The presence of bubbling of holomorphic disks in families of Lagrangians leads to a difficulty in this theory where a discontinuity appears in the disk counts used to construct the GW potential.
In \cite{auroux2007mirror} these discontinuities are explained in terms of a wall-crossing correction which  describes how such bubblings can be appropriately incorporated into the flux coordinates on the space.
For example, the monotone Chekanov and product tori in $\CC^2$ are related by a Lagrangian isotopy which exhibits one of these wall-crossing corrections.

This technique inspired \cite{vianna2014infinitely} to produce examples of non-Hamiltonian isotopic monotone Lagrangians in toric del-Pezzos.
These Lagrangians are constructed by Lagrangian isotopies where wall-crossing occurs; thus the  Lagrangians have related (but not equal) holomorphic disk counts. 
This distinguishes the Hamiltonian isotopy classes of these Lagrangian submanifolds.
A framework for this story was developed by \cite{pascaleff2020wall} , which showed that Lagrangians constructed via Lagrangian mutation (a kind of Lagrangian surgery presented in \cite{haug2015lagrangian}) had disk counts which were related by a wall-crossing transformation.
The examples considered in \cite{vianna2014infinitely} were shown to be constructed via this Lagrangian mutation process.

\subsection{Statement of Main Results}
The goal of this paper is to extend the constructions of \cite{hicks2020tropical} to Lagrangian fibrations $X\to Q$ which are almost toric (and so may admit some fibers with singularities).
In doing so, we shed some light on questions laid out in \cite[section 6.3]{matessi2018lagrangian} regarding the homological mirror symmetry interpretation of monotone tropical Lagrangian tori in toric del-Pezzos. 

This paper first provides an alternate description of the tropical Lagrangian submanifolds from \cite{hicks2020tropical} using the combinatorics of dimers (classically, an embedded bipartite graph $G\subset T^2$).
To a dimer we construct an exact Lagrangian in $(\CC^*)^n$ whose valuation projection lies near a tropical curve (\cref{def:dimerlagrangian},\cref{cor:exact}).
The argument projection $\arg: (\CC^*)^n\to T^n$ of this Lagrangian is related to the dimer initially chosen.
We can find a set of Lagrangian mutations based on the combinatorics of the dimer graph. 

\begin{lemma*}[Dimer-Mutation Correspondance, Restatement of \ref{lemma:dimerantisurgery}]
	Let $L$ be a Lagrangian described by the dimer $G\subset T^2$. 
	Suppose a face $f$ of $G$ has boundary satisfying the zero weight condition (\cref{def:weight}).
	Then we can construct another Lagrangian by mutation,  $\mu_{D_f} L$,  whose argument projection can be explicitly described by another dimer.
\end{lemma*}

This motivates the construction of tropical Lagrangian submanifolds inside of toric del-Pezzos. 
In the case where $\dim_\CC(X)=2$, the singular fibers of a toric fibration $X\to Q$ can be chosen to be of a particularly nice form.
We then call $Q$ an almost-toric base diagram, which has the structure of a tropical manifold. 
We show that tropical curves $V\subset Q$ meeting the discriminant locus of $Q$ admissibly admit Lagrangian lifts $L(V)\subset Q$. 
We then use this to construct some tropical Lagrangian tori in toric del-Pezzos.
As Lagrangian tori, these are interesting because in the complement of an anticanonical divisor they are not isotopic to those constructed in  \cite{vianna2014infinitely}.
\begin{theorem*}[Restatement of \ref{cor:almosttoriclift}, \ref{claim:mutationdirections}]
    Let $X$ be a toric del-Pezzo.
    Let $E\subset X$ be a smooth anticanonical divisor chosen so that there is an SYZ fibration $X\setminus E\to Q$ obtained from pushing in the corners of the Delzant polytope.
    There exists a tropical Lagrangian torus $L_{T^2}\subset X\setminus E$ which is not isotopic to $F_q$, the fiber of the moment map. 
    Furthermore both $L_{T^2}$ and $F_q$ bound matching configurations of Lagrangian antisurgery disks, giving them matching Lagrangian mutations.
\end{theorem*}

The observation that there exists a correspondence between the antisurgery disks with boundary on $L_{T^2}$ and $F_q$  suggests that, although they represent different objects in the Fukaya category, there is autoequivalence of the Fukaya category interchanging the Lagrangians $L_{T^2}$  and $F_q$.
Furthermore, we show that a variation of this construction works more generally whenever one has a certain kind of Lagrangian mutation seed. 

We look to mirror symmetry for why mutation configurations (like those considered by Vianna in toric Fanos) give tropical Lagrangian tori, and restrict to the example of $X=\CP^2$. 
The mirror to $\CP^2\setminus E$ is known to be $\check X_{9111}$, an extremal rational elliptic surface. 
There is an automorphism of $D^b\Coh(\check X_{9111})$ (a fiberwise Fourier-Mukai transform) which interchanges the moduli of points with the moduli of degree 0 line bundles supported on the elliptic fibers. 
Provided that a generation result for the Fukaya category of $\CP^2\setminus E$ is known, we can state what $L_{T^2}$ is as an object of the Fukaya category. 

\begin{theorem*}[Restatement of \ref{cor:nontrivial}, \ref{thm:fmmirror}]
    There exists a symplectomorphism $g:(\CP^2\setminus E)\to (\CP^2\setminus E)$ interchanging $L_{T^2}$ to $F_q$. 
    With \cref{assum:x9111hms},  $L_{T^2}$ is mirror to a line bundle supported on an elliptic fiber of $\check X_{9111}$. 
\end{theorem*}

\subsection{Outline of Construction}
We now outline the rest of this paper, focusing on the construction of the Lagrangian $L_{T^2}$, its surgery disks, and the symplectomorphism $g: \CP^2\setminus E\to \CP^2\setminus E$. \Cref{sec:background} starts with some necessary background and notation.
In \cref{subsec:surgery}, we look at Lagrangian surgery, antisurgery, and mutations. 
These are the tools which we use to build tropical Lagrangian submanifolds and to describe the Lagrangian mutation phenomenon which becomes the focus of inquiry.
\Cref{subsec:toricbackground} reviews tropical geometry on affine manifolds, with an emphasis on dimension two.  

\begin{figure}
    \centering
    \begin{subfigure}{.3\linewidth}
		\centering
		\input{figures/coamoeba.tikz}
		\caption{A dual dimer. The three white hexagons correspond to antisurgery disks for mutation.}
	\label{fig:exampledimer}
		\end{subfigure}\;\;\;
    \begin{subfigure}{.3\linewidth}
        \centering
    \input{figures/examplepicture0.tikz}
    \caption{Two tropical curves related by a nodal trade. These curves have isotopic Lagrangian lifts.}
    \label{fig:example1}
    \end{subfigure}\;\;\;
    \begin{subfigure}{.3\linewidth}
        \centering
        \input{figures/examplepicture.tikz}
    \caption{The Lagrangian $L_{T^2}\subset \CP^2\setminus E$. The center vertex corresponds to \cref{fig:exampledimer}.}
    \label{fig:example2}
    \end{subfigure}
    \caption{}
\end{figure} 

\Cref{sec:dimer} extends the results of \cite{hicks2020tropical} to construct tropical Lagrangian submanifolds from the data of a dimer.
This involves giving a definition for a dual-dimer (\cref{def:dualdimer}) in higher dimensions as a collection of polytopes $\{\Delta^\circ_v\}, \{\Delta^\bullet_w\}$ in the torus $T^n$ whose vertices have a matching condition imposed on them (see \cref{fig:exampledimer}).
We show that such a collection of polytopes corresponds to a tropical hypersurface in $\RR^n$.
In \cref{subsec:dimerlagrangians} we construct from this collection of polytopes a Lagrangian whose valuation projection lies nearby the corresponding tropical hypersurface, and whose argument projection matches the dual dimer. 
\Cref{subsec:dimercomputation} is a slight detour from the main focus of the paper to provide a combinatorial model for the Floer-theoretic support of a tropical Lagrangian in terms of the Kasteleyn operator (similar to the computation in  \cite{treumann2019kasteleyn} for microlocal sheaf theory).

The Lagrangian mutation story is introduced in \cref{subsec:mutations}, where we show that each face of the dimer builds a Lagrangian antisurgery disk on the corresponding Lagrangian. 
These faces arise as sections of the argument projection  over the complement of the polytopes in the dual dimer. 
We additionally show that Lagrangian mutation across these disks can be understood as a modification of the underlying combinatorial dimer. See \cref{fig:exampledimer}. 

In \cref{sec:almosttoric}, we generalize beyond tropical Lagrangians in $(\CC^*)^n$ to tropical Lagrangians in almost toric fibrations.
We show that a tropical curve in an almost toric base diagram has a Lagrangian lift by constructing a local model for the Lagrangian lift near the discriminant locus.
In dimension 2, we prove that deformations of tropical curves lift to Lagrangian isotopies of their Lagrangian lifts. 
\begin{lemma*}[Nodal Trade for Tropical Lagrangians]
	The local models for Lagrangian submanifolds in \cref{fig:example1} are Lagrangian isotopic. 
\end{lemma*}
This lemma becomes a convenient tool for constructing isotopies of Lagrangian submanifolds, and is based on a method in used in \cite{abouzaid2018khovanov} to compare Lagrangians inside of Lefschetz fibrations.
Both the lifting and isotopies of tropical curves are achieved by modeling a node in the almost toric base diagram with a Lefschetz fibration. 
Tropical Lagrangians are described in these neighborhoods as Lagrangian surgeries of Lagrangian thimbles.

In \cref{sec:tdp} we apply the tropical lifting construction from the previous section to build tropical Lagrangian tori in toric del-Pezzos disjoint from an anticanonical divisor (see \cref{fig:example2}). 
The vertex of the tropical Lagrangian is modeled on a dimer. 
A computation shows that the mutation directions of this dimer match the ones known from \cite{vianna2017infinitely,pascaleff2020wall}. 
This constructs the tropical Lagrangian tori from the first theorem. 

Finally, in \cref{subsec:comparingtori} we present an in-depth example of homological mirror symmetry for the example of $\CP^2\setminus E$ following \cite{auroux2006mirror}.
The main observation is that we may choose $E$ to be a member of the Hesse pencil of elliptics, which has a large amount of symmetry. Using this observation, we take $g: \CP^2\setminus E\to \CP^2\setminus E$ to be a pencil automorphism which fixes $E$, but switches its meridional and longitudinal directions. 
The Lagrangian $L_{T^2}$ is compared to a Lagrangian in a neighborhood of $E$ using the mutation and nodal-trade operation for tropical Lagrangians.
It is then observed that $F_q$, a fiber of the SYZ fibration, also may be isotoped so that it too lives near $E$. 
In a neighborhood of $E$, we see that $g$ interchanges these two Lagrangians.
We also present $L_{T^2}$ as a surgery of Lagrangian thimbles which are expected to generate the Fukaya category of $\CP^2\setminus E$, which characterizes the mirror object to $L_{T^2}$ in $D^b\Coh(\check X_{9111})$. 

\subsection{Acknowledgements}

This project wouldn't be possible without the support of my advisor Denis Auroux during my studies at UC Berkeley.

While working on this project, I benefited from conversations with 
    Ailsa Keating, 
    Mark Gross, 
    Diego Matessi, 
    Nick Sheridan
who have been very generous with their time and advice.
I also thank
    Jake Solomon
for providing useful feedback on the introduction to this paper.
Additionally, I would like to thank
    Paul Biran,
who provided me with great amount of mathematical and professional advice during my visit at ETH Z\"urich.
Finally, the exposition of this paper benefitted from the thoughtful feedback of an anonymous reviewer.

A portion of this work was completed at ETH Z\"urich.
This work was partially supported by 
    NSF grants DMS-1406274 and 
    DMS-1344991 
and by a Simons Foundation grant 
    (\#\,385573, Simons Collaboration on Homological Mirror Symmetry).

\section{Some Background}
\label{sec:background}

\subsection{Lagrangian Surgery and Mutations}
\label{subsec:surgery}
Lagrangian surgery is a tool for modifying a Lagrangian along its self intersection locus.
It was introduced by \cite{polterovich1991surgery} in the case where a Lagrangian is immersed with transverse self-intersections. 
In this setting, a neighborhood of the transverse intersection is replaced with a Lagrangian neck.
We will be using two similar notions of surgery. 
One extension is antisurgery along \emph{isotropic surgery disks} \cite{haug2015lagrangian}.

\begin{theorem}[\cite{haug2015lagrangian}] Suppose that $D^k$ is an isotropic disk with boundary contained in $L$ and cleanly intersecting $L$ along the boundary. Then there exists an immersed Lagrangian $\alpha_D(L)\subset X$ called the \emph{Lagrangian antisurgery of $L$ along $D$}, which satisfies the following properties
	\begin{itemize}
		\item
		      $\alpha_D(L)$ is topologically obtained by performing surgery along $D^k$,
		\item
		      $\alpha_D(L)$ agrees with $L$ outside of a small neighborhood of $D^k$,
		\item
		      If $L$ was embedded and disjoint from the interior of $D^k$, then $\alpha_D(L)$ has a single self-intersection point.
	\end{itemize}
\end{theorem}
When we perform antisurgery of an embedded Lagrangian along a Lagrangian disk  $D^{n}$ the resulting Lagrangian has a single self-intersection.\footnote{The notation $\alpha_D(L)$ is chosen as the character $\alpha$ results from applying antisurgery on the character $c$. } There exists a choice of surgery neck so that the resolution of the self-intersection of $\alpha_{D^n}(L)$ by Lagrangian surgery is $L$. However, if we choose a Lagrangian surgery neck in the opposite direction of the disk $D^n$ to combine antisurgery with surgery, we can obtain a new embedded Lagrangian.
\begin{definition}(Adapted from \cite[Definition 4.9]{pascaleff2020wall}).
	Let $L$ be an embedded Lagrangian submanifold, and $D^{n}$ a surgery disk. Let $\alpha_D(L)$ be obtained from $D^{n}$ by antisurgery.  The \emph{mutation of $L$ along $D^n$} is the Lagrangian $\mu_D(L)$  obtained from $\alpha_D(L)$ by resolving the resulting single self-intersection point with the opposite choice of neck.
	\label{def:mutation}
\end{definition}
It is expected that Lagrangians submanifolds which are related by mutation give different charts on the moduli space of Lagrangian submanifolds in the Fukaya category, and that these charts are related by a wall crossing formula \cite{pascaleff2020wall}. A typical example of Lagrangians related by mutation are the Chekanov and Clifford tori in $\CC^2$ obtained by taking two different resolutions of the Whitney sphere.

The second variation of Lagrangian surgery that we use is surgery along a non-transverse intersection with a particular collared neighborhood. 
This surgery replaces two Lagrangians with one in a neighborhood of their symmetric difference.
\begin{prop}
	Let $L_0$ and $L_1$ be two Lagrangians with boundary. Let $U\subset L_0$ be an open neighborhood of $L_0\cap L_1$.
	Suppose there exists a choice of collar neighborhood for the boundary of $U$
	\[
		\partial U\times (0, t_0)_t\subset U
	\]
	and a function $f: U\to \RR$ with the following properties:
	\begin{itemize}
		\item
			  On the collar, $f$ depends only on the $t$ variable, is decreasing and convex, and $f(t_0)=t_0$.
		\item
			  The function vanishes on the complement of $\partial U\times (0, t_0)_t$
		\item
			  In a sufficiently small Weinstein neighborhood  $B^*_{c} U$, the Lagrangian $L_1|_{B^*_c U}$ is the graph of the section $df$.
	\end{itemize}
	Then there exists a Lagrangian $L_0 \#_U^{r,s} L_1$ satisfying the following properties:
	\begin{itemize}
		\item
			  $L_0\#_U^{r,s} L_1$ lives in a small neighborhood of the symmetric difference  $(L_0\cup L_1)\setminus (L_0 \cap L_1))\subset X$. 
		\item
			  There exists a Lagrangian cobordism (in the sense of \cite{biran2013fukayacategories}) $K: (L_0, L_1)\rightsquigarrow L_0\#_U^\epsilon L_1$
	\end{itemize}
	The Hamiltonian isotopy class of $L_0\#_U^{r,s} L_1$ is dependent on the choice of profile functions $r:\RR_{>t_0/2}\to \RR, s: \RR_{>t_0/2}\to \RR$, whose properties are given in \cite[Proposition 3.1]{hicks2020tropical}.
	\label{prop:generalizedsurgeryprofile}
\end{prop} 
The proof is analogous to the proof for the case when $U$ is contractible presented in \cite[Proposition 3.1.]{hicks2020tropical}.
When we only need the Lagrangian isotopy class of $L_0\#_U^{r,s} L_1$, we will drop the decorations of the profile functions and write $L_0\#_U L_1$.
\subsection{Affine and Tropical Geometry}
\label{subsec:toricbackground}

We summarize a description of these tropical manifolds from \cite{gross2011tropical}.
\begin{definition}
	An \emph{integral tropical affine manifold with singularities} is a manifold with boundary $Q$ containing an open subset $Q_0$ such that 
	\begin{itemize}
		\item
			  $Q_0$ is an integral affine manifold with an atlas whose transition functions are in $SL(\ZZ^n)\ltimes \RR^n$. 
		\item
			  $ \Delta:= Q\setminus Q_0$, the \emph{discriminant locus,} is codimension 2
		\item
		      $\partial Q\subset Q$ can be locally modelled after a $SL(\ZZ^n)\ltimes \RR^n$ coordinate change on $\RR^{n-k}\times \RR_{\geq 0}^k.$
	\end{itemize}
\end{definition}
We will be interested in tropical manifolds where the discriminant locus additionally comes with some affine structure. A tropical manifold is a pair $(Q, \mathcal P)$, where $\mathcal P$ is a polyhedral decomposition of $Q$.
For a full definition of the data of a tropical manifold  $(Q, \mathcal P)$, we refer the reader to \cite[Definition 1.27]{gross2011tropical}, and provide a short summary here.
The vertices of this polyhedral decomposition are decorated with fan structures which are required to satisfy a compatibility condition so that the polyhedra may be glued with affine transitions across their faces. 
The compatibility need not extend to affine transitions in neighborhoods of the codimension 2 facets of the polyhedra, giving rise to the discriminant locus, a union of a subset of the codimension 2 faces. 
This determines the affine structure on $Q_0$ completely.
We call such a manifold an integral tropical manifold if all of the polyhedra are lattice polyhedra. 
For most of the examples that we consider, $Q$ will be real 2-dimensional, and the notions of tropical manifold and tropical affine manifold agree with each other.

\subsubsection{Almost Toric Base Diagrams}

The majority of our focus will be in $\dim(Q)= 2$, where there is a graphical notation for describing the affine geometry on $Q$ and correspondingly the symplectic geometry of the 4-dimensional symplectic manifold $X$ \cite{symington71four,leung2010almost}.
To describe the affine structure on $Q$, we describe the monodromy around the singular fibers.
This can be done diagrammatically with the following additional data.

\begin{definition}
	Let $(Q, \mathcal P)$ be a 2-dimensional tropical manifold. Let $Q^0$ be the set of singular points.
	At each point $q_i \in Q^0$ we define the \emph{eigenray} $R_i\subset Q$ to be the ray in the base starting at $q_i$ pointing in the eigendirection of the monodromy around $q_i$.
	A \emph{base diagram} is a map from $Q\setminus \bigcup_i R_i$ to $\RR^2$ with the standard affine structure, with eigenrays marked with a dashed line at each singularity. 
	We decorate the points $q_i$ with the marker $\times_k$, where the monodromy around $q_i$ is a $k$-Dehn twist.
\end{definition}

The Lagrangian fibers $F_q$ of $X\to Q$ can be described by the points in the base diagram. 

\begin{itemize}
	\item
		If a point $q\in Q\setminus \RR$ has a standard affine neighborhood, then $F_q$ is a Lagrangian torus.
	\item
		If the point $q\in Q\setminus \RR$ has an affine neighborhood modelled on $\RR\times \RR_{\geq 0}$ then fiber $F_q$ is an elliptic fiber of corank 1, corresponding to an isotropic circle in $X$.
	\item	
		If the point $q\in Q\setminus \RR$ has an affine neighborhood modelled on $\RR_{\geq 0} \times \RR_{\geq 0}$, the fiber $F_q$ is an elliptic fiber of corank 2, which is simply a point in $X$.
	\item
		If a point $q\in Q\setminus \RR$ belongs to the discriminant locus, then the fiber is a Whitney sphere (if $k=1$) or a plumbing of Lagrangians spheres (if $k> 1$).
\end{itemize}

The \emph{nodal slide, nodal trade,} and \emph{cut transfer} are operations which modify the affine structure of a base diagram $Q$  but correspond to symplectomorphisms of $X\to Q$.
The nodal trade modifies a base diagram by replacing an elliptic corank 2 fiber with a nodal fiber in the neighborhood of an elliptic corank 1 fiber.
This replaces a corner with a nodal fiber whose eigenline points in the balancing direction to the corner. See \cref{fig:nodaltrade}. 
\begin{figure}
	\centering
		\input{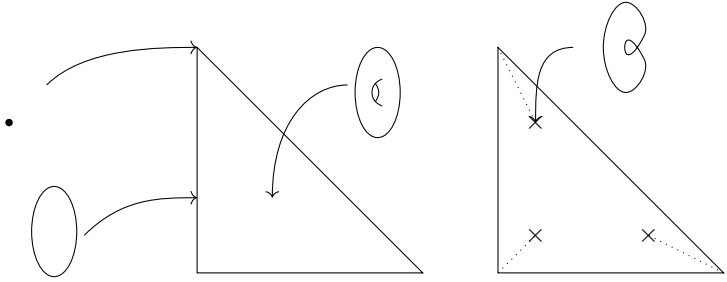}
		\caption
		{
			The nodal trade  applied three times to the toric diagram of $\CP^2$. The toric divisor given by a nodal elliptic curve is transformed into a smooth symplectic torus.
		}
		\label{fig:nodaltrade}
\end{figure}

\subsubsection{Tropical Differentials}
In the setting where $Q=\RR^n$, a tropical hypersurface is defined via the critical locus of a tropical function $\phi: Q\to \RR$.
However, in the general setting of tropical manifolds there are sets which are locally described by the critical locus of tropical functions but cannot be globally described by a tropical function due to monodromy around the singular fibers.
Since the  construction of tropical Lagrangians only requires the differential of the tropical function, this is not problematic.
\begin{definition}
	Let $Q$ be a tropical manifold. The \emph{sheaf of tropical differentials on $Q_0$} is the sheaf $\Omega^1_{\aff}$ on the space $Q_0$. It is given by the sheafification of the quotient: 
	\[
		\Omega^1_{\aff}(U)=\{\phi: U\to \RR \}/ \RR
	\]
	where $\phi: U \to \RR$ is a piecewise linear function satisfying the following conditions:
	\begin{itemize}
		\item 
		$d\phi\in T^*_\ZZ U$ whenever $d\phi$ is defined,
		\item 
		For every point $q\in U$ there exists an integral affine neighborhood $B_\epsilon(q)$ so that the restriction $\phi|_{B_\epsilon(q)}$ is concave.
	\end{itemize}
	The sheaf $\RR$ here is the sheaf of constant functions.
	The sheaf of \emph{integral tropical differentials} is the subsheaf of constant sections of $T^*_\ZZ(Q_0)$.

	Let $i:Q_0\into Q$ be the inclusion.
	We define the \emph{sheaf of tropical sections}\footnote{In \cite{gross2011tropical}, these are called piecewise linear affine multi-valued functions}  to be the quotient sheaf 
	\[
	\dTrop:= i_*(\Omega^1_{\aff})/i_*(T^*_\ZZ Q_0).	
	\]
	\label{def:affinedifferentials}
\end{definition}
We will call the sections of this sheaf the tropical sections, and denote them  $\phi \in \dTrop(U)$.\footnote{This is an abuse of notation, as there may not be a globally defined function whose differential describes this section.
However, this will make the remainder of our discussion consistent with the notation used to construct tropical Lagrangians.} 
Given a tropical section $\phi$, we denote the locus of non-linearity as $V(\phi)\subset Q$. 
Should $\phi$ have a representation in each chart by a smooth tropical polynomial, we say that $\phi$ is smooth.

\begin{remark}
	A point of subtlety: the quotient defining the sheaf of tropical sections is performed over $Q$, not $Q_0$. 
	Importantly, while the presheaves \begin{align*}
		i_*(\Omega^1_{\aff})/_{\pre}i_*(T^*_\ZZ(Q\setminus \Delta))\\
		i_*(\Omega^1_{\aff}/_{\pre}T^*_\ZZ(Q\setminus \Delta))
	\end{align*} agree, their sheafifications do not. 
	In particular, the sheaf of tropical differentials remember that in the neighborhood of the discriminant locus, the tropical section must actually arise from a representative tropical differential.
	\label{rem:subtlety}
\end{remark}
The examples drawn in \cref{fig:ada,fig:adb,fig:adc,fig:affinedifferential2} illustrate how each component of \cref{def:affinedifferentials} is being used, and \cref{fig:affinedifferential3} gives an non-example demonstrating the relevance of \cref{rem:subtlety}
\begin{figure}
	\centering
	\begin{subfigure}{.3\linewidth}
		\centering
		\input{figures/affinedifferentials1a.tikz}
		\caption{}\label{fig:ada}
	\end{subfigure}
	\begin{subfigure}{.3\linewidth}
		\centering
		\input{figures/affinedifferentials1b.tikz}
		\caption{}\label{fig:adb}
	\end{subfigure}
	\begin{subfigure}{.3\linewidth}
		\centering
		\input{figures/affinedifferentials1c.tikz}
		\caption{}\label{fig:adc}
	\end{subfigure}
	\caption*{\Cref{fig:ada}: An affine manifold $Q$ with charts $A, B$ and discriminant locus $\times$.\newline\newline
	\Cref{fig:adb}: Contour plots of piecewise linear concave functions defined over the two charts. The differ on overlaps by constants.\newline\newline
	\Cref{fig:adc}: These two functions give the data of a globally defined tropical differential on $Q$, whose locus of non-linearity is drawn in red.}
	\label{fig:affinedifferential1}
\end{figure} 
\begin{figure}
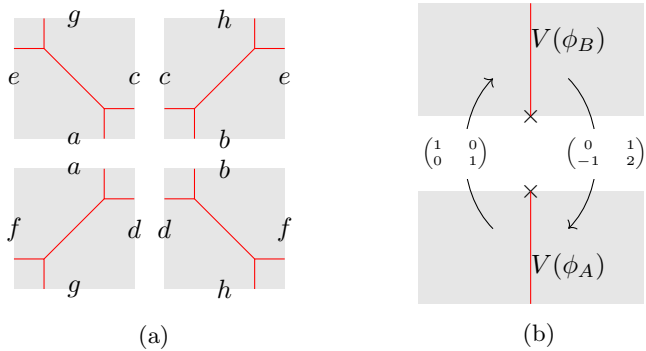

	\centering
	\begin{subfigure}{.4\linewidth}
		\centering
		\input{figures/affinedifferentials2.tikz}
		\caption{}
		\label{fig:affinedifferential2}
	\end{subfigure}
	\begin{subfigure}{.4\linewidth}
		\centering
		\input{figures/affinedifferentials3.tikz}
		\caption{}
		\label{fig:affinedifferential3}
	\end{subfigure}
	\caption*{\Cref{fig:affinedifferential2}: An example of a tropical section defined over the affine torus. The affine torus is covered with 4 charts, and over each chart the tropical section is defined by a tropical differential (whose locus of nonlinearity is drawn in red). Globally, there is no tropical differential with the specified locus of nonlinearity.\newline\newline
	\Cref{fig:affinedifferential3}: A \emph{nonexample} of a tropical section, from \cref{rem:subtlety}. $\phi_A$ and $\phi_B$ define tropical differentials on $A$ and $B$ respectively; after modding out by $T^*_\ZZ Q$ they assemble to a well defined section \emph{outside of the discriminant locus.} 
	However, there is no tropical differential with prescribed locus of nonlinearity on any chart containing the discriminant locus.}
\end{figure}
When $Q=\RR^n$, there is no difference between the global sections of $\dTrop$ and the differentials of global tropical polynomials.

Given a triple $(Q, \mathcal P, \phi)$, one can construct a dual triple  $(\check Q, \check{\mathcal P}, \check \phi)$ using a process called the discrete Legendre transform.
Away from the boundary the base manifolds $Q$ and $\check Q$ agree as topological spaces, however their affine structures differ at the singular points.
At the boundary these spaces are modified so that the non-compact facets of $Q$ are compactified in $\check Q$ and vice-versa.
The simplest example of this phenomenon is when $\check Q= \Delta_\Sigma\subset \RR^2$ is a compact polytope.
The Legendre dual to $\check Q$ is the plane $Q=\RR^2$, equipped with a fan decomposition whose non-compact regions correspond to the boundary vertices of $\check P$.

Given a tropical manifold $Q$, we can produce a torus bundle $X_0=T^*Q_0/T^*_\ZZ Q_0$ over $Q_0$. This space $X_0$ comes with canonical symplectic and almost complex structure arising from the affine structure on $Q_0$. 
In good cases this compactifies to an almost toric fibration $X$ over $Q$.
Similarly, we may produce a associated manifold $\check X$ over $\check Q$. 
The pair of spaces $X$ and $\check X$ are candidate mirror spaces.
When $Q$ is non-compact we expect that $Q$ is equipped with additional data in the form of a monomial admissibility condition or stops in order to obtain a meaningful mirror symmetry statement.
This admissibility condition should be constructed by considering the open Gromov Witten invariants of $\check F_p$.
The computation of these invariants is beyond the scope of our exposition, and we'll be content with constructing our admissibility conditions in an ad-hoc manner.

\subsubsection{Some examples of tropical sections}

A running example that we will use is the symplectic manifold $\CP^2\setminus E$. One can construct an almost toric fibration for $\CP^2\setminus E$ by starting with the toric base diagram for $\CP^2$.
By applying nodal trades at each corner, we obtain a toric fibration $\overline{\val}:\CP^2\to Q_{\CP^2}$, where the boundary of $ Q_{\CP^2}$ is an affine $S^1$ (see \cref{fig:nodaltrade}).
The preimage of $\val^{-1}(\partial  Q_{\CP^2})=E\subset \CP^2$ is a symplectic submanifold isotopic to a smooth cubic. 
$\CP^2\setminus E$ is the total space of an almost toric fibration over  interior of this set, $Q_{\CP^2\setminus E}= Q_{\CP^2}\setminus \partial Q $.  
The monodromy around the three singular fibers allows us to construct some more interesting tropical sections of $Q$.  We give three such examples of these sections and their associated tropical subvarieties below.
\begin{itemize}
	\item
		Tropical sections which have critical locus close to the boundary of $Q_{\CP^2\setminus E}$. 
		\Cref{fig:tpII} gives an example of such a section. 
		Even though the critical locus appears to have three corners, the affine coordinate change across the branch cuts means that this critical locus is actually an affine circle.
	\item
		The example given in \cref{fig:tpIV} is an example of a tropical section which does not arise as the differential of a globally defined tropical function.
		The critical locus terminates at the nodal point, and points in the direction of the eigenray of the nodal point.
	\item
		 Tropical sections which meet the singular fibers coming from admissible tropical sections as in \cref{fig:tpIII}.
		 This gives us an example of a compact tropical curve in $Q$ of genus 1. 
\end{itemize}

\begin{figure}
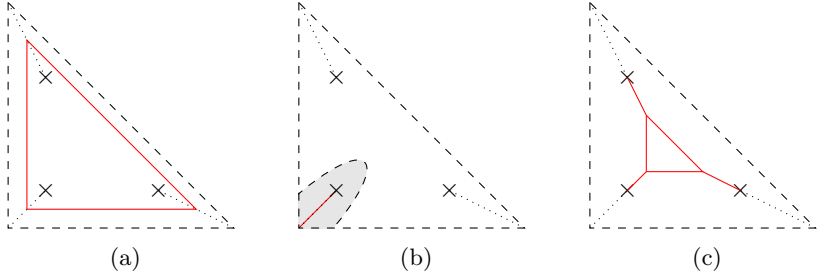

	\centering
	\begin{subfigure}{.3\linewidth}
		\centering
		\input{figures/tropicalpotentialII.tikz}
		\caption{}
		\label{fig:tpII}
	\end{subfigure}
	\begin{subfigure}{.3\linewidth}
		\centering
		\input{figures/tropicalpotentialV.tikz}
		\caption{}
		\label{fig:tpIV}
	\end{subfigure}
	\begin{subfigure}{.3\linewidth}
		\centering
		\input{figures/tropicalpotentialIII.tikz}
		\caption{}
		\label{fig:tpIII}
	\end{subfigure}
	\caption{Tropical subvarieties associated to some tropical sections  on $\CP^2\setminus E$.
	\Cref{fig:tpIV} locally is modelled on \cref{fig:affinedifferential1,fig:affinedifferential2,fig:affinedifferential3}, as is \cref{fig:tpIII} near the discriminant locus.}
\end{figure}

The examples above are typical of the kind of phenomenon which may occur for tropical curves in affine tropical surfaces. 
\begin{definition}
	Let $V\subset Q$ be a tropical curve in an affine tropical surface. 
	We say that $V$ avoids the critical locus if $V$ is disjoint from $\Delta$ and $\partial Q$.  
	We say that the interior of $V$ avoids the critical locus if $V$ is disjoint from $\partial Q$,and at each node $q\in\Delta$,
	there is a neighborhood $B_\epsilon(q)$ so that the restriction of $V\cap B_\epsilon(q)$ is a ray parallel to the eigenray of $q$. 
\end{definition}

\section{Tropical Lagrangians from Dimers}

\label{sec:dimers}
\label{sec:dimer}

We now introduce a combinatorial framework generalizing some of the ideas discussed in \cite[Section 5.2]{matessi2018lagrangian}, and the previous work of \cite{treumann2019kasteleyn,ueda2013homological,shende2019cluster,feng2008dimer}.
\begin{definition}
	A \emph{dimer} is an embedded bipartite graph $G$ on $T^2$ so that $V(G)=V^\circ\sqcup V^\bullet$.
	A \emph{non-embedded dimer} is a bipartite graph $G$ on $T^2$. 
	A zigzag configuration for a dimer $G$ is a set of transverse cycles $\Sigma\subset C_1(T^2)$ satisfying the following conditions:
	\begin{itemize}
		\item
			  Each connected component in $T^2\setminus \Sigma$ contains at most one vertex of $G$.
			  These connected components are called the \emph{dimer faces}, and are indexed by $V(G)$.
		\item
		      Each edge of the dimer is transverse to every cycle. Each edge passes through exactly one intersection point between 2 cycles.
		\item
			The oriented normals of the cycles point outward on the $V^\circ$ dimer faces, and inward on the $V^\bullet$ dimer faces.
	\end{itemize}
\end{definition}
We will now restrict to the setting of \emph{dual dimers}, where $\Sigma$ is a collection of affine cycles.
Denote by $[\Sigma]\subset H_1(T^2)$ the set of homology classes  of cycles in $\Sigma$.
Note that zig-zag configurations $\Sigma$ satisfy a balancing condition $\sum_{[c]\in [\Sigma]} [c]=0$.
It is the case that for every set of homology classes $[\Sigma]\subset H_1(T^2)$ satisfying the balancing condition, we can find a dimer whose zigzag collection $\Sigma$ represents the homology classes $[\Sigma]$ \cite{gulotta2008properly}.
 However, it is not necessarily the case that we can find an \emph{affine} dimer with this property: see \cite[Section 4]{forsgaard2016dimer}.
A dimer picks out an oriented two chain whose boundary is $\Sigma$.
This is similar to the data used in \cite{shende2019cluster}.
\begin{figure}
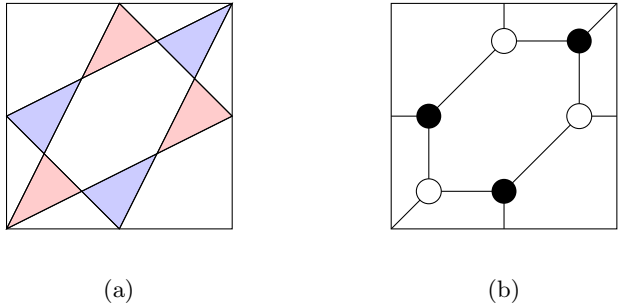

	\centering
	\begin{subfigure}{.24\linewidth}
		\centering
		\input{figures/coamoeba.tikz}
		\caption{}
	\label{fig:hexagonzigzag}
		\end{subfigure}\hspace{2cm}
	\begin{subfigure}{.24\linewidth}
	\centering
	\input{figures/hexagondimer.tikz}
	\caption{}
	\label{fig:hexagondimer}
	\end{subfigure}
		\caption[Dimers and their zigzag graphs]{An example of a dimer and associated bipartite graph.  }
	\end{figure}
More generally, we will consider pairs of the following form:
\begin{definition}
	A \emph{dual $n$-dimer}  is two finite collections of $n$-polytopes 
	\[
		\{\Delta^\circ_v\}, \{\Delta^\bullet_w\}\subset \RR^n
	\] 
	which satisfy the following properties.
	\begin{itemize}
		\item
			  Each vertex set $\{\Delta^\opm_v\}^0$ is a set of distinct points on the torus in the sense that whenever $w_1, w_2\in \{\Delta^\circ_v\}^0$ and $w_1\equiv w_2\mod \ZZ^{n}$,then $w_1=w_2$.
		\item
			  We  require that these two vertex sets match after quotienting by the lattice, 
			  \[
				  \{\Delta^\circ_v\}^0/\ZZ^n=\{\Delta^\bullet_w\}^0/ \ZZ^n.
			\]
		\item
		      Let $p_1\in \Delta^\circ_{v_1}$ be a vertex, and let $p_2\in \Delta^\bullet_{v_2}$ be the corresponding vertex so that $p_1\equiv p_2\mod \ZZ^n$. Let $\{e_{1}, \ldots, e_{k}\}$ be the edges of $\Delta^{\circ}_{v_1}$ containing the vertex $p_1$.
		      We require that the edges of $\Delta^\bullet_{v_2}$ containing $p_2$ point in the opposite directions $\{-e_{1}, \ldots, -e_{k}\}.$
	\end{itemize}
	If the interiors of the $\Delta^\circ_v$ and $\Delta^\bullet_w$ are disjoint $\mod \ZZ^n$, we say that the dual dimer configuration has no self-intersections.

	From this data, we obtain a bipartite graph $G\subset T^{n}$, whose vertices are indexed by  $\{\Delta^\circ_v\}\cup \{ \Delta^\bullet_w\}$, and whose edges are determined by which polytopes in the dual dimer share a common vertex.
	\label{def:dualdimer}
\end{definition} 
We will usually index the polytopes by the vertices $v^{\opm}\in V^\circ\sqcup V^\bullet = V(G)$. 
The edges of the bipartite graph are in bijection with $\{\Delta^\circ_v\}^0=\{ \Delta^\bullet_w\}^0$. 
The graph $G$ need not be embedded. If the polytopes $\{\Delta_v^\opm\}$ are disjoint, then $G$ can be chosen to be embedded.
A dual dimer prescribes the data of a $n$-chain in $T^n$. 
Our requirement that $G$ is bipartite guarantees that this $n$-chain is oriented.

We now briefly explore some of the combinatorics of these dual dimers to produce the data of a tropical hypersurface in $\RR^n$.
\begin{claim}
	The edges of an dual dimer all have rational slope.
\end{claim}
\begin{proof}
Let $e$ be an edge of $\Delta_v^\circ$, with ends on vertices $p_-,p_+\in \{\Delta_v^\opm\}^0$. 
From our definition of a dual dimer, there exists an edge $e_-$ in some $\Delta_{w}^\bullet$ which also has end on $p_-$ and is parallel to $e$. By concatenating $e_-$ and $e_+$, we obtain a line segment.
By repeating this process, we obtain an affine representative of a cycle in $H_1(T^n, \ZZ)$ associated to each edge $e$. 
\end{proof}

\begin{claim}
	Let $\{\Delta^\circ_v\}, \{\Delta^\bullet_w\}$ be a dual $n$-dimer. Let $\alpha$ be a facet of some $\Delta^\bullet_v$. Consider $T^\alpha\subset T^n$, the affine $(n-1)$ subtorus spanned by $\alpha$.
	The set of $(n-1)$ polytopes $\Delta_\beta^\opm$ given by the facets of our original set of polytopes which satisfy
	\begin{align*}
		\{\Delta_\beta^\bullet\;|\; \text{$\beta$ is a facet of $\Delta^\bullet$, $\beta\subset T^\alpha$}\}\\
		\{\Delta_\beta^\circ\;|\; \text{$\beta$ is a facet of $\Delta^\circ$, $\beta\subset T^\alpha$}\}
	\end{align*}
	is the data of an $(n-1)$ dimer on $T^\alpha$. 
\end{claim}
By induction, we get the same result for all faces.
\begin{corollary}
	Let $\alpha$ be a $k$-face of some $\Delta^\bullet_v$. Consider $T^\alpha\subset T^n$, the affine sub-torus spanned by $\alpha$.
	The set of $k$ polytopes given by the $k$-faces satisfying
	\begin{align*}
		\{\Delta_\beta^\bullet\;|\; \text{$\beta$ is a $k$-face of $\Delta^\bullet$, $\beta\subset T^\alpha$}\}\\
		\{\Delta_\beta^\circ\;|\; \text{$\beta$ is a $k$-face of $\Delta^\circ$, $\beta\subset T^\alpha$}\}
	\end{align*}
	is a dual $k$-dimer of $T^\alpha$.
\end{corollary}

Each of these $k$ dimensional dual dimers gives the data of a $k$-chain in $T^\alpha$. 
We denote these $k$-chains of $T^n$,
\[
	\{\underbar{U}^\beta\;|\; \text{$\beta$ is a $k$-face}\}\subset C_k(T^n, \ZZ).
\]
This can also be thought of an equivalence relation on the set of $k$-faces of the dual-dimer, where two faces are equivalent if they define the same dual $k$-dimer chain. 
A \emph{cone} is the real positive span of a finite set of vectors. Given a cone $V\subset \RR^n$, a subspace $U\subset \RR^n$, the $U$-relative dual cone of $V$ is 
\[
	V^{\vee|U}:=\{u\in U \; | \;  \langle u, V\rangle \geq 0\}.	
\]
To each $k$-chain $\underbar{U}^\beta$ we can associate a cone in $\RR^n$.
\begin{definition}
	Let $\underbar{U}^\beta$ be a chain given by a face $\beta\subset \Delta^\opm_v$.
	Assume that we have translated $\Delta^\opm_v$ so that the origin is an interior point of the face $\beta$. 
	Let $\RR^\beta$ be the affine subspace generated by $\beta$.
	Let $(\RR^\beta)^\bot$ be the corresponding perpendicular subspace. 
	We define the dual cone to the face $\underbar{U}^\beta$ to be 
	\[
		\underbar{U}_\beta:=\left\{
			\begin{array}{cl} 
				(\RR_{\geq 0}\cdot \Delta^\bullet_v)^{\vee|(\RR^\beta)^\bot} & \text{ If $\beta$ belongs to a $\bullet$ polytope}\\
				-(\RR_{\geq 0}\cdot \Delta^\circ_v)^{\vee|(\RR^\beta)^\bot}& \text{ If $\beta$ belongs to a $\circ$ polytope}
		\end{array} \right.
	\]
\end{definition}
Suppose that $\alpha$ and $\beta$ are facets in the same dual $k$-dimer so that $\underbar U^\alpha=\underbar U^\beta$.
Let $\alpha\subset \Delta_v^\bullet$, and suppose that $\beta\subset \Delta_w^\bullet$. After translating $\Delta_v^\bullet$ and $\Delta_w^\bullet$  so that $0\in \alpha$ and $0\in \beta$, we get an agreement of the cones $\RR_{\geq 0}\cdot \Delta^\bullet_v = \RR_{\geq 0}\cdot \Delta^\bullet_w$. 
Similarly, if $\gamma\subset \Delta_u^\circ$ and $\underbar U^\gamma=\underbar U^\alpha$, then $\RR_{\geq 0} \cdot \Delta^\bullet_v = -\RR_{\geq 0} \cdot \Delta^\circ_u$. It follows that:
\begin{claim}
	If $\underbar{U}^\alpha\subseteq \underbar{U}^\beta$, then  $\underbar{U}_\alpha\supseteq \underbar{U}_\beta$
\end{claim}
This also shows that the definition of the cone is really only dependent on the data of the $k$-chain represented by the choice of face $\alpha$, in that $\underbar U_\alpha=\underbar U_\beta$ whenever $U^\alpha = U^\beta$.  
Consider the polyhedral complex containing the subset $U^\beta$. This complex satisfies the \emph{zero tension condition}, and therefore describes a tropical subvariety of $\RR^n$. 
We will denote this tropical hypersurface by $V$, and the codimension-$k$ strata of this tropical hypersurface by $V^k$.
\subsection{Dimer Lagrangians}
From the data of a dimer, we now construct a Lagrangian inside of $X=(\CC^*)^n$. 
The construction of these Lagrangians are similar to the construction of tropical Lagrangians in \cite[sections 3.1, 3.2]{hicks2020tropical}.
Let $Q=\RR^n$, and $T^*_\ZZ Q$ be the lattice in the cotangent bundle generated by $dq_1, \ldots, dq_n$.
We give $X$ the symplectic structure via identification with $T^*Q/T^*_\ZZ Q$, and let $\val: X\to Q$ be the valuation projection. The fibers of this projection are Lagrangian tori. 
We denote by $\arg: X\to (T^*)_0Q/(T^*_\ZZ)_0Q=T^n$ the argument projection to a torus fiber. 
The Newton polytope of a piecewise linear function $\phi: Q\to \RR$ is the convex hull of the projection of $\Im(d(\phi))$ to $(T^*)_0Q=\RR^n$, taken wherever the derivative is defined.
\label{subsec:dimerlagrangians}
\begin{definition}
	Let $\Delta_v\subset \RR^n$ be a polytope.
	The \emph{convex dual tropical function} $\phi^\circ_v:Q\to \RR$ is the convex piecewise strictly\footnote{Here, strictly linear means not just that the derivative is constant, but the extension to all of $Q$ sends the origin to $0$. It is worth pointing this out, as the standard convention is to use piecewise linear to mean piecewise affine.} linear function with Newton polytope $\Delta_{v}$.
	Similarly, define $\phi^\bullet_v$ to be the concave dual tropical function, $\phi^\bullet_v=-\phi^\circ_{v}$.
\end{definition}
Given a dual dimer $\{\Delta^\circ_v\},\{ \Delta^\bullet_w\}$, let $\{\phi^\circ_v\},\{ \phi^\bullet_w\}$ be the associated dual tropical functions.
Following \cite{hicks2020tropical}, let  $\tilde\phi^{\opm}_\rho: \RR^n\to \RR$ be smoothings of the convex functions by a kernel $\rho: \RR^n\to \RR$ of small radius $R$.
We add a small constant to this function so that $\tilde\phi^{\opm}_\rho(0)=0$.
\begin{definition}[\cite{abouzaid2009morse}]
	The \emph{tropical Lagrangian section }$\sigma_{\phi, \rho}: Q\to X$ associated to $\phi$ is  the composition
	\[
		\begin{tikzcd}
			T^*Q\arrow{r}{ / T^*_\ZZ Q } & X\\
			Q\arrow{u}{d\tilde \phi_\rho}
		\end{tikzcd}.
	\]	
	\label{def:tropicallagsection}
\end{definition}
For convenience of notation, when the choice of smoothing kernel is unimportant, we suppress it and simply write $\sigma_\phi$. 
Furthermore, given the data of $\{\Delta^\circ_v\},\{ \Delta^\bullet_w\}$, we set 
$\sigma^{\opm}_k:= \sigma_{\phi^{\opm}_k}$. 
We glue together the tropical Lagrangian sections $\sigma^{\opm}_k$ along their overlapping regions. 
\begin{claim}
	Let $\{\Delta^\circ_v\},\{ \Delta^\bullet_w\}$ be a dual dimer configuration without self-intersections.
	There is a decomposition of the intersections of the $\sigma^\circ_{v,\rho}$ and $\sigma^\bullet_{w,\rho}$,
	\[
		\bigcup_{v, w\in G} \sigma_{v,\rho}^\bullet\cap \sigma_{w,\rho}^\circ =\bigcup_{e\in G} U_{e,\rho}.
	\] 
	The argument projection of each component $U_{e,\rho}$ is the shared corner between polytopes $\Delta^\bullet_w, \Delta^\circ_v$.
	Furthermore, the sections $\sigma^{\opm}_{k,\rho}$ have intersections with collared boundaries in the sense of \cref{prop:generalizedsurgeryprofile} at each of the $U_{e,\rho}$.
\end{claim}
The structure of a collared boundary on the intersections $U_{e,\rho}$ follows from the convexity/concavity of the primitive functions $\phi^{\opm}_{k,\rho}$. 
If the dual dimer has self-intersections, there will be possibly be additional intersections between the $\sigma^{\opm}$ which overlap in the argument projection.
We will explore the effect of these additional intersections in \cref{exam:immersed,subsec:mutations}.

We now examine the intersection locus $U_{e,\rho}$ associated to an edge  $e=vw$ in more detail. 
This set $U_{e,\rho}$ describes one of the regions on where $\tilde \phi^\bullet_{w,\rho}$ is linear. 
Since $\tilde \phi^\bullet_{w,\rho}$ is a smoothing of $\phi^\bullet_w$, $U_{e, \rho}$ approximates one of the regions of linearity of $\phi^\bullet_w$, where the accuracy of this approximation can be characterized by the radius $R$ of the smoothing kernel $\rho$. 
Let  $\beta$ be the common vertex of the two dimer polytopes corresponding to the edge $e$. 
\cite[Proposition 3.11]{hicks2020tropical} characterizes the relationship between $U_{e, \rho}$ and the tropical hypersurface by:
\[U_{e,\rho}=\{q\in \underbar U^\beta \;|\; B_R(q)\subset \underbar U^\beta\}.\]
It is worth noting that the boundary $\partial U_{e, \rho}$ is a level set of $\phi^\bullet_w$.
Since $\phi^\bullet_w$ matches $\tilde \phi^\bullet_{w,\rho}$ over $U_{e, \rho}$, the boundary of $\partial U_{e, \rho}$ is also a level set of the smoothed tropical primitive. 
The height of this level set
\begin{equation}
	\tilde \phi^\bullet_{\rho, r}|_{\partial U_{e, \rho}} \propto R
	\label{eq:levelatneck}
\end{equation} 
is linearly proportional to the smoothing radius $R$ chosen.
Because by construction $\tilde\phi^{\opm}_{\rho, r}(0)=0$, the constant of proportionality is nonzero.

\begin{definition}
	Let $\{\Delta^\circ_v\},\{ \Delta^\bullet_w\}$ be a dual dimer. 
	Let $\mathcal D:=\{\rho, \{r_e, s_e\}_{e\in E}\}$ be a choice of smoothing parameter, and surgery profile function for each $e\in E$.
	The \emph{dimer Lagrangian} is the Lagrangian connect sum
	\[
		L^{\mathcal D}(\phi^\bullet_w, \phi^\circ_v):=  \sigma^\circ_{v,\rho} \underset{U_e\;|e\in G,}{\#^{r_e, s_e}} \sigma_{w,\rho}^\bullet.
	\]
	\label{def:dimerlagrangian}
\end{definition}
\Cref{fig:1dimexample} gives an example of these dimer Lagrangians, where the dimer considered is the 4-cycle on $S^1$, lifting to a Lagrangian $S^1\subset \CC^*$ by surgering together 4 tropical sections.
Because  $U_{e,\rho}$ approximates  $\underbar{U}_\beta$, the valuation of a dimer Lagrangian is approximates the tropical hypersurface $V$ associated to the dimer. 
By the argument projection property of \cite[Theorem 3.17]{hicks2020tropical} the surgery of Lagrangian sections does not change the argument projection.
\begin{claim}
	The image of $\arg (L^{\mathcal D}(\phi^\bullet_w, \phi^\circ_v))$ is $\{\Delta^\bullet_w, \Delta^\circ_v\}\subset T^n$. 
	\label{claim:argument}
\end{claim}

\begin{figure}
	\centering
	\input{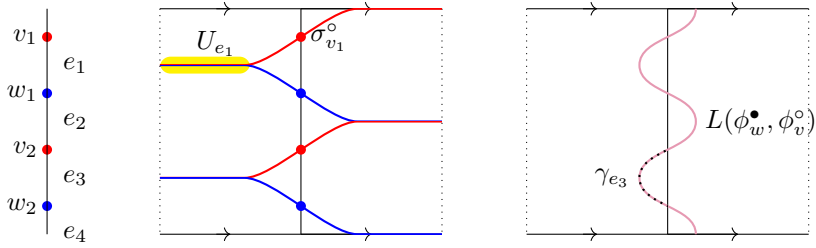}
	\caption{A dimer on $S^1$, the corresponding set of Lagrangian sections on $\CC^*$, and the resulting dimer Lagrangian obtained after surgering the overlap regions. }
	\label{fig:1dimexample}
\end{figure}
Different choices of surgery profile data $\mathcal D$ lead to Lagrangian isotopic (but not Hamiltonian isotopic) Lagrangian submanifolds. 
We now show that there is a preferred Hamiltonian isotopy class of these dimer Lagrangian submanifolds, corresponding to a choice of data making $L^{\mathcal D}(\phi^\bullet_w, \phi^\circ_v)$ exact.
\begin{definition}
	Let $L^{\mathcal D}(\phi^\bullet_w, \phi^\circ_v)$ be a dimer Lagrangian. Let $G$ be the associated graph.
	Give $G$ the structure of a directed graph with edges going from $\bullet$ to $\circ$.
	To each edge $e$, let 
	 \[\gamma_e : [0,1]\to L^{\mathcal D}(\phi^\bullet_w, \phi^\circ_v)\]
	 be a lift of the edge $e$ to the dimer Lagrangian.  
	We define the \emph{weight} of an edge $e$ to be the integral
	\[
		\w_e:= \int_{\gamma_e}\eta,
	\]
	where $\eta=q\cdot dp$ is the tautological 1-form on the cotangent bundle of $T^n$.\footnote{We apologize for the tragic conflict of notation here. We're forced to use $q$ as the coordinate on $Q$, the base of the SYZ fibration, and $p$ for the coordinate on $T^*_0Q/T^*_\ZZ Q$, the fiber of the SYZ fibration. Unfortunately, when we want to treat this as the cotangent bundle of the fiber, our choices make $q$ the fiber coordinate on $T^*T^n$, $p$ the coordinate on the base $T^n$, and $\eta=q\cdot dp$ the canonical 1-form.} 
	\label{def:weight}
\end{definition}
The weight of a cycle $c\subset E(G)$ is 
\[
	\sum_{e\in c} \sgn(e, c)\w_e
\]
which descends to a map on  $\w: H_1(L^{\mathcal D}(\phi^\bullet_w, \phi^\circ_v))\to \RR$ measuring non-exactness of the dimer Lagrangian.
The sign $\sgn(e, c)$ is $+1$ if the $\bullet\to \circ$ orientation of $e$ agrees with $c$, and $-1$ if the orientation of $e$ disagrees with $c$.
\begin{claim}
	Let $\mathcal D_1, \mathcal D_2$ be two choices of surgery profile data. The weight maps $\w_1: H_1(L^{\mathcal D_1} (\phi^\bullet_w, \phi^\circ_v))\to \RR$ and $\w_2: H_1(L^{\mathcal D_2}(\phi^\bullet_w, \phi^\circ_v))\to \RR$ agree if and only if there exists an isotopy between  $L^{\mathcal D_1} (\phi^\bullet_w, \phi^\circ_v))$ and $L^{\mathcal D_2} (\phi^\bullet_w, \phi^\circ_v))$ with total flux zero.
\end{claim}
\begin{proof}
	Since the space of surgery data is connected, $L^{\mathcal D_1} (\phi^\bullet_w, \phi^\circ_v))$ and $L^{\mathcal D_2} (\phi^\bullet_w, \phi^\circ_v))$ are Lagrangian isotopic, identifying the homology groups  $H_1(L^{\mathcal D_i} (\phi^\bullet_w, \phi^\circ_v))$.
	Since we are in the symplectically exact setting, the flux of a Lagrangian isotopy is equal to the change in the evaluation of the flux primitive on homology, i.e. $\w_2(c)-\w_1(c)$.
\end{proof}

The weight of an edge $e$ encodes the flux swept  by the choices of surgery neck and smoothing at region $U_{e,\rho}$ in the construction of $L^{\mathcal D}(\phi^\bullet_w, \phi^\circ_v)$.
\begin{lemma}
	For any assignment of $\{\mathfrak v_e\}_{e\in E} \in \RR_{\geq 0}$ of weights, there exists a choice of smoothing kernel $\rho$ and surgery profiles $r_e, s_e$ so that $\w_e=\mathfrak v_e$ for all $e$.
	\label{lemma:weightchoices}
\end{lemma}
\begin{figure}
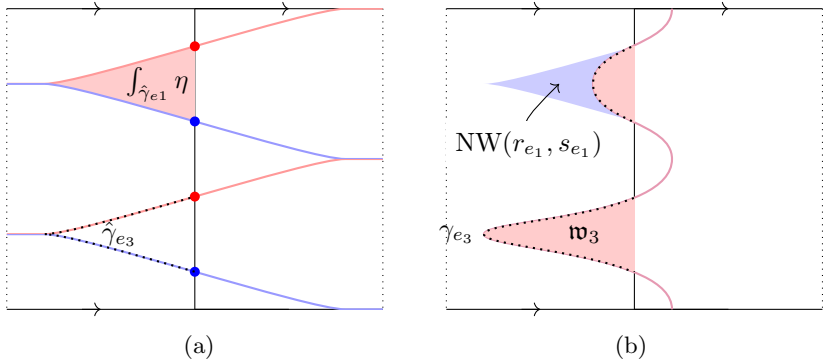

	\centering
	\begin{subfigure}{.45\linewidth}
	\centering
	\input{figures/modifiednecks1.tikz}
	\caption{}
	\label{fig:beforesurgery}
	\end{subfigure}
	\begin{subfigure}{.45\linewidth}
	\centering
	\input{figures/modifiednecks2.tikz}
	\caption{}
	\end{subfigure}
	\caption{Modifying the parameters $\mathcal D$ change the Lagrangian isotopy class of the Lagrangian (in comparison to \cref{fig:1dimexample}). On the left, increasing the smoothing parameter $\rho$ increases $\w_e$ before applying surgery. On the right, choosing large surgery necks decreases the weights $\w_e$.}
	\label{fig:modifiednecks}
\end{figure}
\begin{proof}
	We give a sketch  before performing the computation in detail.
	The vertices of the dimer $v^\opm$ can be rearranged in such a way that $\arg(d\tilde \phi^\opm_v(0))=v^\opm$. 
	With this choice of endpoints, the valuation projection of a path $\val(\gamma_e)$ starts at the origin of $Q$, travels out to the surgery neck $U_{e,\rho}$, and then comes back to the origin.
	The weight $\w_e$ is roughly proportional to the length of this path, which is controlled by how close the surgery neck comes to the origin.
	This can be increased for all surgery necks by taking a larger radius for $\rho$, and decreased at a specific surgery neck by picking surgery profile data $(r_e, s_e)$ with larger neck radius. See \cref{fig:modifiednecks}, in comparison to \cref{fig:1dimexample}.

	 Let $R$ be the radius of the support of the smoothing kernel $\rho$. 
	 Consider sections $\sigma^\bullet_{w, \rho}, \sigma^\circ_{v, \rho}$ which overlap over $U_{e,\rho}$. 
	There is a lifting of the edge $e$ to a path  $\hat \gamma_e$, which has the property that $\hat\gamma_e\subset \sigma^\bullet_w\cup \sigma^\circ_v$ and $\arg(\gamma_e)=e$ (drawn as a dotted path in \cref{fig:beforesurgery}). 
	We show that $\int_{\hat \gamma_e}\eta$ is dependent on the radius of the smoothing kernel.
	Parameterize $\hat \gamma_e: [0, 1]\to (\CC^*)^n$   so that  $\hat \gamma_e|_{[0, 1/2]}\subset \sigma^\bullet_w$, and $\hat \gamma_e|_{[ 1/2,1]}\subset \sigma^\circ_v$.
	We note that $q_e(1/2)\in \partial U_{e,\rho}$, whose properties are described by \cref{eq:levelatneck}.

	The section $\sigma^\bullet_{w,\rho}= d\tilde \phi^\bullet_{w,\rho}$ is parameterized in $(q,p)$ coordinates by  $( q^i,  \partial_{q_i} \tilde \phi^\bullet_{w,\rho} )$.
	Write the path $\hat \gamma_e(t)=( q_e^i(t), p_e^i(t))_{i=1}^n$ so that 
	\[p_e^i(t)=\left\{\begin{array}{cc}\left.\partial_{q_i} \tilde \phi^\bullet_{w,\rho}\right|_{q_e(t)} & t<\frac{1}{2}\\
		\left.\partial_{q_i} \tilde \phi^\circ_{v,\rho}\right|_{q_e(t)}  & t>\frac{1}{2}
	\end{array}\right.\]
	is a parameterization of the edge $e$. 
	The integral of $\eta$ along this path is:
	\begin{align*}
		\int_{\hat \gamma_e}\eta = \sum_{i=1}^n \int_{0}^1& q_e^i \cdot \frac{d}{dt}(p_e^i) = \sum_{i=1}^n (q_e^i \cdot p_e^i)|_{t=0}^{t=1} - \int \frac{dq_e^i}{dt}\cdot  p_e^i dt\\
		\intertext{as $q_e^i=0$ for $t=0, 1$}
		=& -\int_{0}^{1/2} d\tilde \phi^\bullet_{w,\rho}\left(\frac{dq_e}{dt}\right) dt-\int_{1/2}^{1} d\tilde \phi^\circ_{v,\rho} \left(\frac{dq_e}{dt}\right) dt\\
		=& \tilde \phi^\bullet_{w, \rho}(q_e(0))-\tilde \phi^\bullet_{w, \rho}(q_e(1/2))+\tilde \phi^\circ_{v,\rho}(q_e(1/2))-\tilde \phi^\circ_{v,\rho}(q_e(1))\\
		\intertext{Without loss of generality, we can assume that  $\rho$ is an approximation of the identity. This means that there exists a small constant $\epsilon_\rho>0$ so that $|\tilde \phi^\opm_{k, \rho}(q_e(0))-\phi^\opm_k(q_e(0))|<\epsilon_\rho$. Since $q_e(0)=0$, $q_e(1)=0$and $\phi^\opm_k(0)=0$}
		>&-\tilde \phi^\bullet_{w, \rho}(q_e(1/2))+\tilde \phi^\circ_{v,\rho}(q_e(1/2))-  2\epsilon_\rho
	\end{align*}
	As $\tilde \phi^\circ_{v,\rho}$ (resp. $\tilde \phi^\bullet_{w,\rho}$) is convex (resp. concave,) and the value of $\tilde \phi^\opm_{k,\rho}$ on $\partial U_{e, \rho}$ is positive (resp. negative) and governed by \cref{eq:levelatneck}.
	In summary, we can make $\int_{\hat {\gamma}_e}\eta$ as large as desired by increasing the radius $R$.

	The discussion following \cite[Figure 2]{hicks2020tropical} assigns a \emph{neck width} $\text{NW}(r_e, s_e)$ to each choice of surgery profile $(r_e, s_e)$.
	This quantity measures the flux swept by the surgery in the sense that:
	\[\w_e= \int_{\gamma_e}\eta=\int_{\hat {\gamma}_e}\eta  -\text{NW}(r_e, s_e).\] 
	The neck width is bounded by the value of the primitive $f$ from \cref{prop:generalizedsurgeryprofile} over the surgery region.
	In our case, the primitive $f$ is $-\tilde \phi^\bullet_{w, \rho}+ \tilde \phi^\circ_{v,\rho}$, so we can choose profile functions making $\text{NW}(r_e, s_e)$ to be as close to $-\phi^\bullet_{w, \rho}(q_e(1/2))+\phi^\circ_{v,\rho}(q_e(1/2))$ as desired. 

	To conclude the proof: for a fixed set of weights $\{\mathfrak v_e\}_{e\in E}$, choose $R$ large enough so that $\int_{\hat {\gamma}_e}\eta> \mathfrak v_e$ for all $e$, and then pick surgery profiles so that 
	\[\text{NW}(r_e, s_e)=\left(\int_{\hat {\gamma}_e}\eta\right)  -\mathfrak v_e.\]
\end{proof}
\begin{corollary}
	There exists a choice of surgery data $\mathcal D$ so that $\w:  H_1(L(\phi^\bullet_w, \phi^\circ_v))\to \RR$ is constantly zero.
\end{corollary}
\begin{proof}
	By \cref{lemma:weightchoices}, we can choose the weight at every edge to be a fixed constant $A>0$.
	Since the graph $G$ is bipartite, every cycle $c$ transverses an equal number of edges in the $\bullet\to\circ$ direction as $\circ\to \bullet$ direction, so that $\w_c=\sum_{w^\bullet v^\circ\in c} \w_{w^\bullet v^\circ}+\sum_{v^\circ w^\bullet\in c} \w_{v^\circ w^\bullet}=0$.
\end{proof}
\begin{corollary}
	There exists a choice of surgery data ${\mathcal D}$ making $L^{\mathcal D}(\phi^\bullet_w, \phi^\circ_v)$ exact. 
	\label{cor:exact}
\end{corollary}

When we write  $L(\phi^\bullet_w, \phi^\circ_v)$, we will always mean an exact dimer Lagrangian.
By choosing surgery data with weights going to zero, we obtain Lagrangians whose valuations approximate $V$, the tropical hypersurface associated to the dimer.
Recall that $V^k$ is the codimension $k$ strata of $V$. 
Let $N^*V^0/N^*_\ZZ V^0\subset (\CC^*)^n$ the collection of conormal Lagrangian tori to the top dimensional strata.
\begin{claim}
	 There exists a Lagrangian $L$, Hamiltonian isotopic to $L(\phi^\bullet_w, \phi^\circ_v)$, which fibers over the tropical hypersurface $V$ in the complement of a neighborhood of $V^1$, that is,
	\[L|_{X\setminus \val^{-1}(B_\epsilon(V^1))}= (N^*V^0/N^*_\ZZ V^0)|_{X\setminus \val^{-1}(B_\epsilon(V^1))}.\]
	\label{claim:codimension0matches}
\end{claim}
\begin{proof}
	The Lagrangian submanifolds $L(\phi^\bullet_w, \phi^\circ_v)$ and  $N^*V^0/N^*_\ZZ V^0$ are isotopic in the complement of $\val^{-1}(B_\epsilon(V^1))$; the flux of this isotopy is zero as they are both exact. 
	Since $L(\phi^\bullet_w, \phi^\circ_v)$ and  $N^*V^0/N^*_\ZZ V^0$ are $C^1$ close, there is a Hamiltonian isotopy between  $L(\phi^\bullet_w, \phi^\circ_v)|_{X\setminus \val^{-1}(B_\epsilon(V^1))}$ and $(N^*V^0/N^*_\ZZ V^0)|_{X\setminus \val^{-1}(B_\epsilon(V^1))}$.
	By interpolating the Hamiltonian over the remainder of $L(\phi^\bullet_w, \phi^\circ_v)$, we obtain the desired Lagrangian submanifold $L$. 
\end{proof}

\begin{example}
	Consider the dual dimer model drawn in \cref{fig:mutationhexagons}. 
	The six triangles drawn are associated to the following six tropical functions.
	\begin{align*}
		\phi_1^\circ =   & (x_1^{6/6}\odot x_2^{6/6})\oplus (x_1^{5/6}\odot x_2^{4/6})\oplus (x_1^{4/6}\odot x_2^{5/6})\\
		\phi_2^\circ =   & (x_1^{2/6}\odot x_2^{4/6})\oplus (x_1^{0/6}\odot x_2^{3/6})\oplus (x_1^{1/6}\odot x_2^{2/6})\\
		\phi_3^\circ =   & (x_1^{4/6}\odot x_2^{2/6})\oplus (x_1^{2/6}\odot x_2^{1/6})\oplus (x_1^{3/6}\odot x_2^{0/6})\\
		\phi_1^\bullet = &-(x_1^{0/6}\odot x_2^{0/6})\oplus (x_1^{1/6}\odot x_2^{2/6})\oplus (x_1^{2/6}\odot x_2^{1/6})\\
		\phi_2^\bullet = &-(x_1^{2/6}\odot x_2^{4/6})\oplus (x_1^{3/6}\odot x_2^{6/6})\oplus (x_1^{4/6}\odot x_2^{5/6})\\
		\phi_3^\bullet = &-(x_1^{4/6}\odot x_2^{2/6})\oplus (x_1^{5/6}\odot x_2^{4/6})\oplus (x_1^{6/6}\odot x_2^{3/6})\\
	\end{align*}
	All six functions give the same nonlinearity stratification to $Q$,
	\[
		V (\phi_i^\bullet)=V(\phi_i^\circ)= V( x_1\oplus x_2\oplus (x_1\odot x_2)^{-1}).
	\]
	There are nine Lagrangian surgeries that we need to perform in order to build $L(\phi^\bullet_w, \phi^\circ_v)$.
	The valuation projection of the Lagrangian submanifold approximates the tropical curve with three legs. 
	\label{exam:ellipticchecker}
\end{example}
These dimer Lagrangians serve as a generalization of tropical Lagrangians constructed in \cite{hicks2020tropical}, where
\[
	L(\phi)= L(\phi_v, 0).
\]
\begin{example}
	One can also assemble lifts of more complicated tropical curves by gluing several dimer Lagrangians together. 
	For example, the genus 1 tropical curve drawn in \cref{fig:smoothtorus} can be built from taking three vertices.
	At each vertex we place a dimer whose cycles are normal to the edges of the vertices. 
\end{example}
\begin{figure}
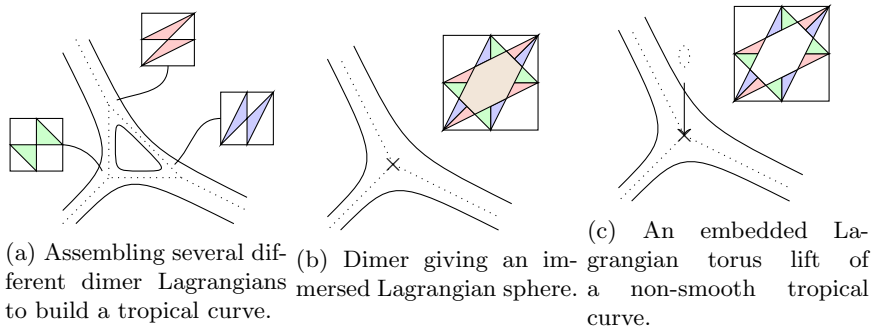

	\centering
	\begin{subfigure}{.3 \linewidth}
		\centering
		\input{figures/matessismoothing.tikz}
		\caption{Assembling several different dimer Lagrangians to build a tropical curve.}
		\label{fig:smoothtorus}
	\end{subfigure}
	\begin{subfigure}{.3 \linewidth}
		\centering
		\input{figures/matessismoothing2.tikz}
		\caption{Dimer giving an immersed Lagrangian sphere.}
		\label{fig:mutatedhexagon}
	\end{subfigure}
	\begin{subfigure}{.3 \linewidth}
		\centering
		\input{figures/matessismoothing3.tikz}
		\caption{An embedded Lagrangian torus lift of a non-smooth tropical curve.}
		\label{fig:mutationhexagons}
	\end{subfigure}
	\caption
	{
		Three related Lagrangians 
	}
	\label{fig:toricdimermutation}
\end{figure}
\begin{example}
	It is not necessary for the dimer model to consist of disjoint faces.
	In \cref{fig:mutatedhexagon} we see a configuration with two triangular faces which overlap at a hexagon. 
	The Lagrangian associated to this dimer is immersed, but has the same legs as the example in \cref{fig:mutationhexagons}. 
	\label{exam:immersed}
\end{example}
The relation between these three examples is that of Lagrangian mutation, which we will formalize to other examples in \cref{subsec:mutations}.

\subsection{Floer Theoretic Support from Dimer Model }
\label{subsec:dimercomputation}
We now restrict ourselves to the setting  $(\CC^*)^2= T^*F_0$
and describe a combinatorial approximation of 
\[
	\CF(L(\phi^\bullet_w, \phi^\circ_v), F_0),
\]
the Floer theory of our tropical Lagrangian against fibers of the SYZ fibration.
\begin{definition}
	Let $\{\Delta^\circ_v\},\{ \Delta^\bullet_w\}$ be a dual dimer configuration with affine bipartite graph $G$.
	Let $\del$ be a $\CC^*$ connection $T^n$, assigning to each path $e:v\to w$ an element $\del_e(1)\in \CC^*$, the image of 1 under parallel transport.
	The Kasteleyn complex with weighting $\del$ is the 2-term chain complex $C^\bullet(G, \del)$ which as a graded vector space is  
	\[ 
		\left(\bigoplus_{v_i^\circ\in V^\circ}\CC\langle v_i^\circ\rangle\right) \oplus \left(\bigoplus_{w_j^\bullet\in V^\bullet}\CC\langle w_j^\bullet\rangle[1]\right)
	\]
	 The differential $d_\nabla$ is determined by the  structure coefficients
	\[
		\langle d_\nabla^\Sigma(v^\circ), w^\bullet\rangle =\sum_{\substack{e\in E(G)\\ e=v^\circ w^\bullet}} \del_e(1).
	\]
	The \emph{support} of $\{\Delta^\circ_v\},\{ \Delta^\bullet_w\}$ is the set of local systems
	\[
		\Supp(\{\Delta^\circ_v\},\{ \Delta^\bullet_w\}):= \{\nabla\;|\; H^1(G, \del)\neq 0\}.
	\]
\end{definition}
In dimension 2, $G$ is exactly a dimer, and the support is the zero locus of the polynomial
\[
	Z^G(\nabla):= \det(d_\nabla). 
\]

The terminology comes from literature on dimers \cite{kenyon2006dimers}. 
By letting the local system $\del$ determine a weight for each edge of the dimer, the terms of the determinant corresponds to the product of weights of a maximal disjoint set of edges (called its Boltzmann weight). 
A maximal disjoint set of edges in a dimer is called a \emph{dimer configuration}, and the sum of Boltzmann weights over all configurations gives the partition function $Z^G(\nabla)$ of the dimer.

We now explain the relation between the  Kasteleyn complex $C^\bullet(G, \del)$ and the Lagrangian intersection Floer complex $CF(L(\phi^\bullet_w, \phi^\circ_v), (F_0, \del))$.
These complexes are isomorphic as vector spaces, as the intersection points of $F_0$ and $L(\phi^\bullet_w, \phi^\circ_v)$ are in bijection with the vertices of the dimer.
The Lagrangian $L(\phi^\bullet_w, \phi^\circ_v)$ is built from taking a surgery of the pieces $\sigma_{v^\opm}$. 
An expectation from \cite{fukaya2010Lagrangian} is that holomorphic strips contributing to the differential $\mu^1: CF(L_0\#_p L_1, L_2)$ are in correspondence with  holomorphic triangles contributing to $\mu^2: CF(L_0, L_1)\tensor CF(L_1, L_2)$.
In our construction of $L(\phi^\bullet_w, \phi^\circ_v)$ we smoothed regions larger than intersection points between the sections $\sigma_{v}^\opm$, however we expect a similar result to hold. 
These intersections are in correspondence with the edges of the dimer $G$, and so we predict that the differential on $CF(L(\phi^\bullet_w, \phi^\circ_v), (F_0, \del))$ should be given by weighted count of edges in the dimer.
The local system $\del$ on $F_0$ determines the weight of the holomorphic strips corresponding to each edge.
\begin{conjecture}
	The isomorphism of vector spaces 
	\[
		\CF(L(\phi^\bullet_w, \phi^\circ_v), F_0)\to C^\bullet(G, \del)
	\]
	is a chain homomorphism.
\end{conjecture}
If this conjecture holds, we have a new tool for computing the support of the Lagrangian $L(\phi^\bullet_w, \phi^\circ_v)$, which will be determined by the zero locus of $Z^G(\nabla)$. 
\begin{example}
	A first example to look at is the Kasteleyn complex of \cref{exam:ellipticchecker}. We give the polygons of the dimer the labels from  \cref{exam:ellipticchecker}. 
	We can rewrite $Z^G(\nabla)$ as a polynomial by picking coordinates on the space of connections.
	Let $z_1$ and $z_2$ be the holonomies of a local system $\del$ along the longitudinal and meridional directions of the torus.
	The differential on the complex $C^\bullet(G, \del)$ in the prescribed coordinates is
	\[
		d_\nabla^\Sigma=
		\begin{pmatrix}
			z_2^\frac{1}{3}       & (z_1z_2)^\frac{-1}{3} & z_1^\frac{1}{3}       \\
			(z_1z_2)^\frac{-1}{3} & z_1^\frac{1}{3}       & z_2^\frac{1}{3}       \\
			z_1^\frac{1}{3}       & z_2^\frac{1}{3}       & (z_1z_2)^\frac{-1}{3}
		\end{pmatrix}.
	\]
	The determinant of $d_\nabla^\Sigma$ is  $Z^G(z_1, z_2)=3-(z_1+z_2+\frac{1}{z_1z_2})$.
	\begin{figure}
		\centering
		\input{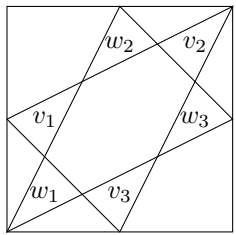}
		\caption{The labelling of faces for the dimer model}
		\label{fig:ellipticfano}
	\end{figure}
	\label{exam:dimerelliptic}
\end{example}

This polynomial is a reoccurring character in the mirror symmetry story of $\CP^2$; for example, it is the superpotential $\check W_\Sigma$ determining the mirror Landau-Ginzburg model. This computation motivates \cref{sec:tdp}.
\begin{remark}
	In the above definition, we've assumed that $\mathcal D$ has been chosen in such a way that the weights $\w_e$ are negligible and can be ignored. 
	More generally, one should look use the Novikov ring in place of $\CC$, and let $\del$ be a Novikov unitary connection. 
	The differential becomes
	\[\langle d_\nabla^\Sigma(v^\circ), v^\bullet\rangle =\sum_{\substack{e\in E(G)\\ e=v^\circ v^\bullet}} \del_e(1) T^{\w_e}.\]
	For generic choice of weights $\w_e$, there are no choices of local systems with $\det(d_\nabla)=0$; for instance, a necessary condition is that the minimal weight must show up at least twice. 
\end{remark}

\subsection{Mutations of Tropical Lagrangians}
\label{subsec:mutations}
In previous examples, we exhibited different dimer models with the same associated  tropical curve.
We now describe how the different Lagrangian lifts of these dimers are related to each other in dimension 2. 
\begin{lemma}
	Let $\{\Delta^\circ_v\},\{ \Delta^\bullet_w\}$ be a dimer model with graph $G$.
	For each face $f\in F(G)$, let $c=\partial_f$ be the boundary cycle of the face.
	Suppose that $c$ has zero weight.
	Let $\gamma_c: S^1\to  L(\phi^\bullet_w, \phi^\circ_v)$ be a lift of the cycle to the dimer Lagrangian, in the sense that 
	\[
		\arg(\gamma_c)=c.
	\]
	There exists a Lagrangian disk $D_f$ with $\partial D_f=c\subset L(\phi^\bullet_w, \phi^\circ_v)$.
	\label{lemma:disks}
\end{lemma}
\begin{proof} 
	Let $V_f\subset T^2$ be the subset of the Lagrangian torus $T^2\subset T^*T^2$ corresponding to the face $f$. 
	Write $\gamma_c(\theta)=(q(\theta), p(\theta))$, where $q(\theta)\in T^*_{p(\theta)}T^2$.\footnote{We again apologize for the inconvenience of using $(q, p)$ for coordinates on $T^*Q/T^*_\ZZ Q$.}
	Let $u:S^1_\theta\times [-\epsilon, \epsilon]_r\to T^2$ be a coordinates on a normal neighborhood of $\gamma_c$, so that we may write $q(\theta)=q_\theta(\theta)d\theta+ q_r(\theta) dr$. 
	The zero weighting condition tells us that 
	\[
		\int_{\gamma_c}\eta= \int_{S^1} q_\theta(\theta)d\theta=0
	\]
	so there exists a primitive $\alpha_\theta: S^1\to \RR$ so that $q_\theta d\theta= d\alpha_\theta$. 
	Let $\rho:[-\epsilon, \epsilon]_r\to \RR$ be a bump function which is constantly 1 in a neighborhood of $r=0$.
	Consider the function $\alpha:=\rho\cdot (\alpha_\theta+r\cdot q_r): T^2 \to \RR$.
	We can compute that $d\alpha= d\rho\cdot (\alpha_\theta+r\cdot q_r) + \rho \cdot \left(d\alpha_\theta+ q_r dr + r\cdot dq_r\right)$. 
	Since $\rho|_c=1, d\rho|_{c}=0$ and $r|_{c}=0$, the pullback of $d\alpha$ to $c$ is $d\alpha_\theta+q_rdr=q$. 
	Therefore, $q$ can be extended to an exact 1-form on $T^2$. 
	The Lagrangian disk $D_f$ is defined by the graph of this 1-form. 
\end{proof}
The Lagrangian antisurgery  $\alpha_{D_f}L(\phi^\bullet_w, \phi^\circ_v)$  is an immersed Lagrangian, which we now describe with a dual dimer model.
Let $\partial f:=\{v^\bullet_1, v^\circ_1, \ldots, v^\bullet_k, v^\circ_k \}$ be the sequence of vertices of $G$ corresponding to the boundary of $f$. 
Recall that $\Sigma$ is the set of cycles in $T^2$ given by the boundary polygons of the dual dimer model.
Let $\Im(\Sigma)\subset T^2$ be the image of these cycles.
After taking an isotopy of $c$, we may assume that $\arg(c)\subset \Im(\Sigma)$. 
We can also require that $\arg(c)$ is a homeomorphism onto its image.
We now take a parameterization 
\[
	h: S^1\times [-1, 1] \to L(\phi^\bullet_w, \phi^\circ_v)
\]
 for a neighborhood  of  $\gamma_c\subset L(\phi^\bullet_w, \phi^\circ_v)$, with $h( \theta,0)=\gamma_c(\theta)$. 
 The boundary components of the collar $h:S^1\times [-1, 1]$ give two cycles in $L(\phi^\bullet_w, \phi^\circ_v)$, which we will label
 \begin{align*}
	\gamma_c^\bullet:=& h(\theta, -1)\\
	\gamma_c^\circ:=& h(\theta, 1)
 \end{align*}
 The path $\gamma$ has argument contained within $\Sigma$, but we require  the map $h(\theta, t):S^1\times [-1, 1]\to L(\phi^\bullet_w, \phi^\circ_v)$ have argument  
\begin{align*}
	\Im(\arg\circ \gamma_{c^\bullet})\subset \Sigma\cup \{\Delta^\bullet_w\}\\
	\Im(\arg\circ \gamma_{c^\circ  })\subset \Sigma\cup \{\Delta^\circ_v\}
\end{align*}
which ``alternates'' between bleeding into argument of $L(\phi^\bullet_w, \phi^\circ_v)$. By \cref{claim:argument}, this is the collection of $\Delta^\circ_v$ and $\Delta^\bullet_w$ polytopes. See \cref{fig:twiststrip}. 
\begin{figure}
	\centering
	\input{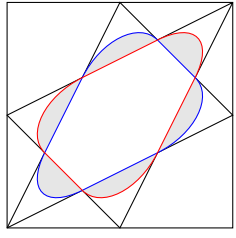}
	\caption{The strip $h$ is parameterized to twist into the  $\Delta^\circ_v$ and $\Delta^\bullet_w$ polytopes.}
	\label{fig:twiststrip}
\end{figure}
We now state this alternating condition.
We require at each $\theta$ exactly one of the three following cases occur:
\begin{itemize}
	\item 
	That the $\circ$ component bleeds out of $\Sigma$ into the interior of the dimer so $\arg\circ  h(\theta, 1)\not\in \Sigma$ 
	\item
	That the $\bullet$ component bleeds out of $\Sigma$ into the interior of the dimer so $\arg\circ h(\theta, -1)\not\in \Sigma$
	\item 
	Neither boundary component bleeds out of $\Sigma$, but the collar $h$ passes through the vertex connected to polytopes in our dimer model so $\arg\circ h(t, \theta)$ maps to a vertex of the   $\Delta^{\opm}_i$. 
\end{itemize}

After performing the Lagrangian surgery, the band parameterized by $h$ will be replaced with two disks $D^\bullet_f$ and $D^\circ_f$.
The boundaries of $D^{\opm}_f$ are the cycles $\gamma_{c^{\opm}}$.

The disk $D^\bullet_f$ glues the polygons $\Delta_{v_i}^\bullet$ which lie along the cycle $\gamma_{c^\bullet}$ to each other.
Similarly, the disk $D^\circ_f$ connects the $\Delta^\circ_{w_i}$ together. 
In summary, the polygons in the cycle $c$ are replaced with two larger polygons in the antisurgery:
\begin{align*}
		\Delta_{f^\bullet}:= & \text{Hull}_{v_i^\bullet\in \partial f} (\Delta_v^\bullet) \\
		\Delta_{f^\circ}:=   & \text{Hull}_{v_i^\circ\in \partial f}( \Delta_w^\circ).
\end{align*}
\begin{lemma}
	Consider a dimer model $\{\Delta^\circ_v\},\{ \Delta^\bullet_w\}$.
	Let $f$ be a face of $G$. 
	Suppose that the boundary of $f$ has zero weight.
	The antisurgery  $\alpha_{D_f} L(\phi^\bullet_w, \phi^\circ_v)$ is again described by a higher  dimer model, whose polygons are given by the collections
	\begin{align*}
		\{\Delta_v^\bullet \;|\; \text{ for all $v^\bullet \not\in \partial f$} \} \cup \{\Delta_{f^\bullet}\}	                     \\
		\{\Delta_w^\circ   \;|\; \text{ for all $w^\circ   \not\in \partial f$} \} \cup \{\Delta_{f^\circ }\}                     \\
	\end{align*}
	\label{lemma:dimerantisurgery}
\end{lemma}
The graph for this dimer is immersed.
For example, from right to left, the Figures \ref{fig:mutationhexagons}, \ref{fig:mutatedhexagon}, and \ref{fig:smoothtorus} describe the antisurgery of a dimer Lagrangian by first obtaining an immersed Lagrangian, then smoothing out the resulting immersed Lagrangian to build an embedded tropical Lagrangian. 
The corresponding dimers are drawn in Figure \ref{fig:immerseddimer}. 
Figure \ref{fig:reallyimmerseddimer} is an immersed dimer, with two vertices of degree three lying atop each other.

\begin{figure}
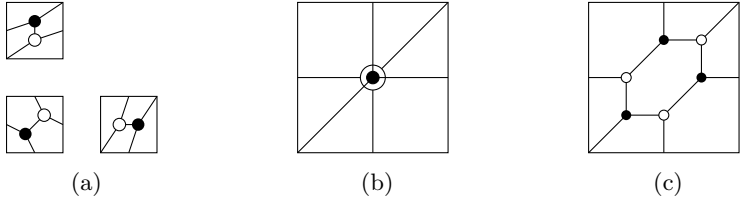

	\centering
	\begin{subfigure}{.3\linewidth}
		\centering
		\input{figures/dimerelliptic.tikz}
		\caption{}
	\end{subfigure}
	\begin{subfigure}{.3\linewidth}
		\centering
		\input{figures/dimerelliptic2.tikz}
		\caption{}
		\label{fig:reallyimmerseddimer}
	\end{subfigure}
	\begin{subfigure}{.3\linewidth}
		\centering
		\input{figures/dimerelliptic3.tikz}
		\caption{}
	\end{subfigure}
	\caption[An immersed dimer on the torus]{The dimers associated to Figure \ref{fig:toricdimermutation}. The middle figure is an immersed dimer with two vertices of degree three, and has a resolution to an embedded dimer on the right.}
	\label{fig:immerseddimer}
\end{figure}
\subsubsection{Seeds and surgeries}
Besides using antisurgery to modify Lagrangian submanifolds, we may use the presence of antisurgery disks for  $L(\phi^\bullet_w, \phi^\circ_v)$ to construct a Lagrangian seed in the sense of  \cite{pascaleff2020wall}.
\begin{definition}[\cite{pascaleff2020wall}]
	A \emph{Lagrangian seed} $(L, \{D_i\})$ is a monotone Lagrangian torus $L\subset X$ along with a collection of antisurgery disks $\{D_i\}$ for $L$ with disjoint interiors, and an affine structure on $L$ making $\partial D_i$ affine cycles. 
	Should the $\partial D_i\subset L$ be the edges of an affine zigzag diagram, we say that this seed gives a dimer configuration on $L$.
\end{definition}
Whenever we have an mutation seed giving a dimer configuration on $L$, we can build a dual Lagrangian using the same surgery techniques used to construct tropical Lagrangians.
We start by taking a Weinstein neighborhood $B_\epsilon^*L$ of $L$.
Let $\{\Delta^\circ_v\},\{ \Delta^\bullet_w\}$ be the dimer model on $L$ induced by the Lagrangian seed structure. 
Using \cref{def:dimerlagrangian}, we can construct  $L(\{\phi^\circ_v\},\{ \phi^\bullet_w\})$ in the neighborhood $B_\epsilon^*L$.
The boundary of $L(\{\phi^\circ_v\},\{ \phi^\bullet_w\})$ is contained in the $\epsilon$-cotangent sphere $S^*_\epsilon L$ and consists of the   $\epsilon$-conormals Legendrians $N^*_\epsilon(\partial D_i)$.
After taking a Hamiltonian isotopy, the disks $\{D_i\}$ can be made to intersect $S^*_\epsilon L$ along $N^*_\epsilon(\partial D_i)$. 
By gluing the dimer Lagrangian to these antisurgery disks, we  compactify $L(\{\phi^\circ_v\},\{ \phi^\bullet_w\})\subset B_\epsilon^*L$ to a Lagrangian $L^*\subset X$.

\begin{definition}
	Let  $(L, \{D_i\})$  be a Lagrangian seed giving a dimer configuration on $L$. We call the Lagrangian  $L^*\subset X$ the dual Lagrangian to $(L, \{D_i\})$. 
	\label{def:dualdimerlagrangian}
\end{definition}

One way to interpret this construction is that a Lagrangian seed has a small symplectic neighborhood which may be given an almost toric fibration. 
The dual Lagrangian $L^*$ is a compact tropical Lagrangian built inside of this almost toric fibration.

By \cref{lemma:dimerantisurgery} the Lagrangian $L^*$ possesses a set of antisurgery disks given by the faces of the dimer graph on $L$. 
Should the antisurgery disks $D_f$ with boundary on $L^*$ form a mutation configuration, we call $(L^*,\{D_f\})$ the dual Lagrangian seed. 

\begin{remark}
	The geometric portion of this construction does not require $L$ or $L^*$ to be tori, although statements about mutations of Lagrangians from \cite{pascaleff2020wall} and relations to mirror symmetry use that $L$ is a torus.
\end{remark}

\section{Tropical Lagrangians in Almost Toric Fibrations}

\label{sec:almosttoric}

Much of the machinery we have constructed for building Lagrangians lifts of tropical hypersurfaces in the fibration $(\CC^*)^n\to \RR^n$ carries over to building tropical Lagrangian hypersurfaces for almost toric fibrations $X\to Q$ with the dimension of the base $\dim Q=2$.
In \cref{subsec:syzthimbles} we look at the local model of a node in a almost toric base diagram and show that lifts of tropical curves can be constructed for tropical curves with edges meeting the singular strata of $Q$ along the eigendirection.
\Cref{subsec:tropicalpantslefschetz} continues using local models from the node based on Lefschetz fibrations to show that isotopy of tropical curves ``through'' a node of the base extend to isotopies of the Lagrangian lifts.
\subsection{Lifting to Tropical Lagrangian Submanifolds}
\label{subsec:syzthimbles}
Recall that $X\to Q$ is an almost toric fibration, $X^0\to Q^0=Q\setminus \Delta$ is an honest toric fibration in the complement of the discriminant locus $\Delta$.  
By abuse of notation, when we are given a tropical section $\phi\in \dTrop(U)$ where $U\subset Q^0$, we will write $\sigma_\phi:U \to X|_U$ to mean the Lagrangian section defined over the bundle $X|_U\to U$ given by some choice of smoothing parameter (see \cite{hicks2020tropical}).
It is immediate that we can use the existing surgery lemma to build tropical Lagrangians away from the critical locus. 
\begin{claim}
	Let $\val:X\to Q$ be an almost toric Lagrangian fibration.
	Suppose that $ V(\phi)\subset Q$ is a (smooth) tropical curve which is disjoint from the critical locus. 
	Then there exists an (embedded) Lagrangian submanifold $L(\phi)\subset X$ whose valuation projection lies in a small neighborhood of $V$. 
	Furthermore, if $Q$ has no boundary, there exists a tropical section $\phi$ so that $L(\phi)=\sigma_0\#\sigma_{-\phi}$. 
\end{claim}
Note that in the non-smooth setting, the number of self-intersection points of the Lagrangian $L(\phi)$ is determined by the sum of the tropical genera  of the non-smooth vertices.
Smoothing the tropical curve by deformation is equivalent to removing the self intersection by Lagrangian surgery. 
In the almost toric case ($\dim(Q)=2$), we can find a Lagrangian lift when the interior of $V$ avoids the critical locus. 
This is built on the following local model, whose monodromy was described in \cite{symington71four}.
\begin{claim}
	Let $X=\CC^2\setminus \{z_1z_2=1\}$ be the symplectic manifold with symplectic fibration 
	\begin{align*}
		W: \CC^2\setminus \{z_1z_2=1\}\to       & \CC\setminus\{1\}    \\
		(z_1, z_2)\mapsto & z_1z_2
	\end{align*} 
	and let $\val: X\to Q$ be the almost toric Lagrangian fibration described in \cite[Section 5.1]{auroux2007mirror}. 
	Then $Q$ has a single node $q_\times$ of multiplicity 1, and there exists a tropical Lagrangian lift of the eigenray of $q_\times$. 
	\label{claim:tropicalthimbles}
\end{claim}

\begin{proof}
	The claim follows from considering the construction of the almost toric fibration arising from the Lefschetz fibration $W$. 
	The rotation $(z_1, z_2) \mapsto (e^{i\theta} z_1, e^{-i\theta}z_2)$ is a global Hamiltonian $S^1$ symmetry which preserves the fibers of the fibration. 
	Let $\mu:(\CC^*)^2\to \RR$ be the moment map of this Hamiltonian action, which also descends to a moment map $\mu: W^{-1}(z)\to \RR$. 
	This map gives an SYZ fibration on the fibers of the Lefschetz fibration. 

	The base of the Lefschetz fibration $\CC\setminus\{1\}$ comes with a standard SYZ fibration by circles $1+r e^{2\pi i \theta}$.
	The symplectic parallel transport map given by the Lefschetz fibration preserves the SYZ fibration on $W^{-1}(z)$;
	as a result, one can build an SYZ fibration for the total space $\{\CC^2\setminus z_1z_2=1\}$ by taking the circles $\val_{W^{-1}(z)}^{-1}(s)$  and parallel transporting them along circles $1+re^{ i \theta}$ of the second fibration to obtain Lagrangian tori 
\[
	F_{r, s} = \{(z_1, z_2)\;|\; |z_1z_2-1|=r, \mu(z_1, z_2)=s\}.
\]
The nodal  degeneration occurs from parallel transport of vanishing cycle through the path $1+ e^{ i \theta}$. 
This corresponds to the single almost toric fiber of this fibration, a Whitney sphere, which occurs in the base when $q_\times=(1, 0).$
$Q$ comes with an affine structure by identifying the cotangent fiber at $q$ with $H^*(F_q, \RR)$, and taking the lattice to be the integral homology classes. 
The monodromy of this fibration around the Whitney sphere acts by a Dehn twist on the vanishing cycle (i.e. for $s=0$) of $F_q$. 
As a result, the coordinate $s$ is a global affine coordinate on $Q$ near $q_x$, but $r$ is not. 
The eigenray is $s=0$. 
 The Lagrangian tori $F_q$ with $q$ in the eigenray of $q_\times$ are those tori which are built from parallel transport of the vanishing cycle.
 See \cref{fig:lefschetzandtropical} for the correspondence between Lagrangians in the Lefschetz fibration and almost toric fibration. 
 \begin{figure}
	\centering
	\input{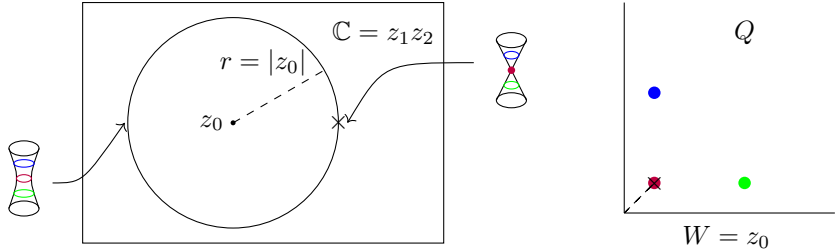}
\caption[Lagrangian tori constructed from a Lefschetz fibration]
{
Lagrangian tori constructed from a Lefschetz fibration giving an almost toric fibration.
The colored fibers correspond to cycles $\ell$ being parallel transported around a circle in the base. 
}
\label{fig:lefschetzandtropical}
\end{figure}
 We now consider the Lagrangian thimble $\tau$ drawn from the critical point $(z_1,z_2)=(0,0)$.
 As the Lagrangian thimble is a  built from a parallel transport of the vanishing cycle, it only intersects the Lagrangians $F_q$ with $q$ on the eigenray of $q_\times$. 
 Therefore, this Lagrangian thimble has valuation projection travelling in the eigenray direction of $q_\times$, proving the claim.
\end{proof}
\begin{corollary}
	Let $\val:X\to Q$ be an almost toric Lagrangian fibration over an integral tropical surface $Q$. 
	Let $V$ be a smooth tropical variety whose interior avoids the discriminant locus $\Delta$. Then there exists a tropical Lagrangian lift $L\subset X$ of $V$.
	\label{cor:almosttoriclift}
\end{corollary}
\begin{proof}
	First, construct the lift of $V$ to a Lagrangian  $\mathring L$ on $X\setminus X^0$. 
	Let $V^0$ be the collection of top dimensional strata of $V$, and $V^1$ the set of vertices of $V$.
	By \cref{claim:codimension0matches} we may take a Hamiltonian isotopy of $\mathring L$ (which we still denote as $\mathring L$) so that $\mathring L|_{X\setminus (\val^{-1}(B_\epsilon(V^1)))}= N^* V^0/N^*_\ZZ V^0|_{X\setminus (\val^{-1}(B_\epsilon(V^1)))}$.
	In particular, the discriminant locus is disjoint from $B_\epsilon (V^1)$, and so  $\val:\mathring L\to Q^0$ honestly fibers over the tropical curve $V$ near the discriminant locus.

	It remains to compactify $\mathring L$ to a Lagrangian submanifold of $X$. 
	At each point $q_i\in X^0$, we take a neighborhood $B_i$ of $q_i$ and model it on the standard neighborhood from \cref{claim:tropicalthimbles}. 
	The portion of $\mathring L$ with valuation over $B_i$ is a Lagrangian cylinder given by the periodized conormal to the eigenray of $q_i$. 
	Similarly, the thimble $\tau_i$ restricted to this valuation is a Lagrangian cylinder given by the periodized conormal to the eigenray of $q_i$. 
	Therefore, we may compactify $\mathring L$ to a Lagrangian $L\subset X$ by gluing the thimbles $\tau_i$ to $L$ at each nodal point such that $q_i\in V$. 
\end{proof}

This allows us to build tropical Lagrangian lifts of the tropical curves described in \cref{fig:tpII,fig:tpIV,fig:tpIII}. 
We may generalize the examples of compact Lagrangian tori in $\CP^2$ to more toric symplectic manifolds with $\dim_\CC(X)=2$.
Let $X_\Sigma$ be a toric surface, and let $\val: X_\Sigma\to Q^{dz}_\Sigma$ be the standard moment map projection. 
The moment polytope $Q^{dz}$ is an example of an almost toric base diagram. 
Consider the almost toric fibration $\val: X_\Sigma\to Q_\Sigma$ obtained by applying a nodal trade to each corner of the moment polytope. The boundary of $Q$ is now an affine circle, corresponding to a symplectic torus $E\subset X_\Sigma$.
\begin{example}
	The neighborhood of $\partial Q_\Sigma$ is topologically  $\partial Q_\Sigma\times [0, \epsilon)_t$.
	For fixed real constant $0< r< \epsilon$, we construct the tropical function $r\oplus t$, which only has dependence on collar direction $t$. 
	This extends to a tropical function over $Q_\Sigma$, whose critical locus is an affine circle pushed off from the boundary $\partial Q_\Sigma$. 
	The critical locus is a tropical curve which avoids the discriminant locus, so there is an associated Lagrangian torus $L_r^{\partial Q_\Sigma}\subset X_\Sigma$ corresponding to this tropical curve. 

	This Lagrangian torus can also be constructed without using the machinery of Lagrangian surgery.
	Let $\gamma\subset E$ be a curve.
	There is a neighborhood $D$ of $E\subset X_\Sigma$ which is a  disk bundle $D\to E$.
	There is a standard procedure to take $\gamma$ and lift it to a Lagrangian $\partial D_\gamma$, the union of real boundaries of this disk bundle along the curve $\gamma$. 
	See \cref{exam:outertropical}.
\end{example}

As one increases the parameter $r$, the Lagrangian $L_r^{\partial Q_\Sigma}$ approaches the critical locus $\Delta_\Sigma$. 
One can continue this family of Lagrangian submanifolds past the critical locus. 

\begin{example}
	In the above example, each nodal point $q_i$ corresponds to a corner of the Delzant polytope $Q^{dz}_\Sigma$.
	The index $i$ is cyclically ordered by the boundary of the Delzant polytope. 
	Let $\Sigma_i$ be the fan generated by vectors $v^-_i, v^+_i$ given by the edges of the corner corresponding to $q_i$. 
	Let $v^\lambda_i$ be the eigenray of $q_i$. 
	Then $\Sigma_i\cup \{v^\lambda_i\}$ is a balanced fan. 
	At each nodal point $q_i$, consider the tropical pair of pants with legs in the directions $\Sigma_i\cup\{v^\lambda_i\}$. 

	The legs of adjacent pairs of pants (from the cyclic ordering) match so that $v^-_i=-v^+_{i+1}$.
	This means that if the pairs of pants are properly placed (say so that the distance from the vertex of the pair of pants along the eigenray direction to the boundary $\partial Q_\Sigma$ are all equal) these assemble into a tropical curve.

	This is a tropical curve whose interior is disjoint from the critical locus, and thus lifts to a tropical Lagrangian with the topology of a torus in $X_\Sigma$. 
	See \cref{exam:innertropical}
\end{example}

\begin{figure}
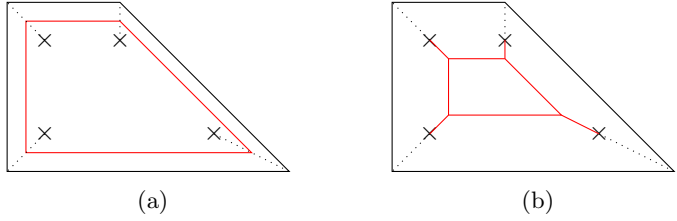
\centering
	\begin{subfigure}{.4\linewidth}
		\centering
		\input{figures/cliffordtropical.tikz}
		\caption{}
		\label{exam:outertropical}
	\end{subfigure}
	\begin{subfigure}{.4\linewidth}
		\centering
		\input{figures/chekanovtropical.tikz}
		\caption{}
		\label{exam:innertropical}
	\end{subfigure}
	\caption{Some more examples of tropical sections}
\end{figure}

\subsection{Nodal Trade for Tropical Lagrangians}
The tropical curves from \cref{exam:innertropical,exam:outertropical} are related via an isotopy of tropical curves.
We now introduce some notations for Lagrangians in Lefschetz fibrations which will allow us to show that this isotopy of tropical curves can be lifted to their corresponding tropical Lagrangians. 
\label{subsec:tropicalpantslefschetz}
The local model for the nodal fiber in an almost toric base diagram is built from a Lefschetz fibration. 
The goal of this section is to build some geometric intuition for interchanging these two different perspectives. 
We now describe three Lagrangian submanifolds which will serve as building blocks in Lefschetz fibrations, similar to those considered in \cite{biran2017cone}. See \cref{fig:buildingblock}.

The first piece is suspension of Hamiltonian isotopy.
Given a path $e:[0, 1]\to \CC$ avoiding the critical values of $W: X\to \CC$, and Hamiltonian isotopic Lagrangians $\ell_0$ and $\ell_1$ in $W^{-1}(e(0))$ and $W^{-1}(e(1))$, we can create a Lagrangian $L^\ell_e$ which is the suspension of Hamiltonian isotopy along the path $e$. 
The image of this suspension under $W$ sweeps out some area in the base of the Lefschetz fibration related to the Hofer norm of the isotopy. 
We work with Hamiltonian isotopies small enough so that the projection of their suspensions avoids the critical values.
This Lagrangian has two boundary components, one above $e(0)$ and one above $e(1)$.
In practice, we will simply specify the Lagrangian $\ell_0$ and assume that the Hamiltonian isotopies are negligible.

The second building block that we consider are the \emph{Lagrangian thimbles}, which are the real downward flow spaces of critical points in the fibration.
These can also be characterized by taking a path $e:[0, 1]\to \CC$ with $e(0)$ a critical value of $W:X\to \CC$, and letting $\ell$ be a vanishing cycle in $W^{-1}(e(1/2))$ for a critical point in $W^{-1}(e(0))$.
The Lagrangian thimble, also denoted $L^\ell_e$, has single boundary component above $e(1)$.

The third building block we will use comes from Lagrangian cobordisms. In any small contractible neighborhood $U\subset \CC$ which does not contain a critical value of $W: X\to \CC$, we can use symplectic parallel transport to trivialize the fibration so it is $W^{-1}(p)\times D^2$ for some $p\in U$.
We then consider cycles $\ell_{1}, \ell_2, \ell_3\subset W^{-1}(p)$ so that $\ell_1\#\ell_2=\ell_3$ with neck size $\epsilon$.
\cite{biran2015lagrangian} constructs a trace cobordism (there, called a $Y$-surgery) of the Lagrangian surgery between these three cycles in the space $W^{-1}(p)\times \CC$.
Given paths $e_1, e_2, e_2\subset D^2$ indexed in clockwise order, with $e_i(1)=p$, we let $L_{e_i}^{\ell_i}$ be the trace cobordism of the surgery between the $\ell_i$ with support living in a neighborhood of the edges $e_i$.
This Lagrangian has three boundary components, which live above $e_i(0)$.
\begin{figure}
	\centering
	\input{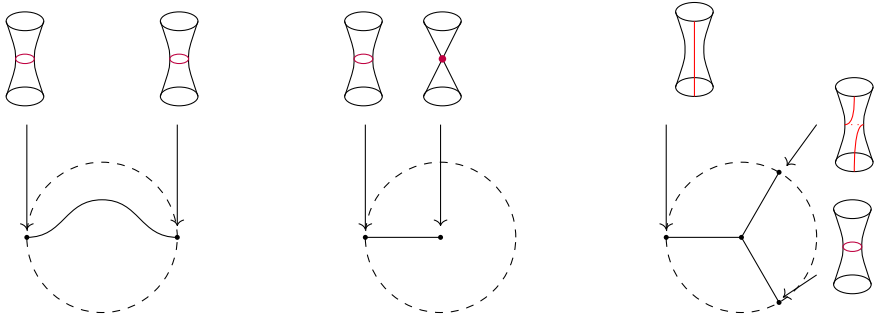}
	\caption[Building blocks for Lagrangian gloves]
	{
		The three different building blocks for a Lagrangian glove: parallel transport, thimbles, and trace of a surgery. 
	}
	\label{fig:buildingblock}
\end{figure}

These pieces glue together to assemble smooth Lagrangian submanifolds of $X$ whenever the ends of the pieces (determined by their intersection with the fiber) agree with each other. 
\begin{definition}
	Let $W: X\to \CC$ be a symplectic fibration.
	A \emph{Lagrangian glove} $L\subset X$ is a Lagrangian submanifold so that for each point $z\in \CC$, there exists a neighborhood $U\ni z$ so that $W^{-1}(U)\cap L$ is one of the three building blocks given above. 
\end{definition}
The reason that we look at Lagrangian gloves is that they can be specified by the following pieces of data:
	\begin{itemize}
		\item
			  A planar graph $G\subset \CC$.
			  This graph is allowed to have semi-infinite edges and loops.
		\item
		    A Lagrangian submanifold $\ell_e\subset W^{-1}(e(0))$ labelling each edge $e\in G$.
	\end{itemize}
This data will correspond to a Lagrangian glove if it satisfies the following conditions:
	\begin{itemize}
		\item
		      The interior of each edge is disjoint from the critical values of $W$.
		\item
		      Outside of a compact set, the semi-infinite edges are parallel to the positive real axis.
		\item
		      All vertices of $G$ have degree $1$ or degree $3$.
		\item
		      Every vertex of degree 1 must lie at a critical value. Furthermore,  the incoming edge $e$ to the vertex $v$ is labelled with a vanishing cycle of the corresponding critical fiber.
		\item
		      Every vertex of degree 3 with incoming edges $e_1, e_2, e_3$ must have corresponding Lagrangian labels $\ell_1, \ell_2$ and $\ell_3$ which satisfy the relation $\ell_1\# \ell_2=\ell_3$ for a surgery of neck size small enough that there exists a disk $D\supset v$ containing the trace of this surgery.
	\end{itemize}
	Such a collection of data gives us a Lagrangian $L^{\ell_e}_G\subset X$.
	
We will diagram these Lagrangians by additionally picking a  choice of branch cuts $b_i$ for $\CC$ so that $W: (X\setminus W^{-1}(b_i))\to (\CC\setminus\{b_i\})$ is a trivial fibration. We can then consistently label the edges of the graph $G\subset \CC$ with Lagrangians in $\ell_e\in W^{-1}(p)$ for some fixed non-critical value $p$.  
Graph isotopies which avoid the critical values correspond to isotopic Lagrangians; furthermore, as long as the label of an edge does not intersect the vanishing cycle of a critical value, we are allowed to isotope an edge over a critical value.

There is another type of isotopy which comes from interchanging Lagrangian cobordisms with Dehn twists \cite{mak2018dehn,abouzaid2018khovanov}, which we now describe. 
Let $v$ be a trivalent vertex with edges $e_1=vw_1, e_2=vw_2, e_3=vw_3$. 
Suppose that the degree of $w_2$ is one (so that $w_2$ is a critical value).
Suppose additionally that the Lagrangians $\ell_1$ and $\ell_2$, the labels above $e_1$ and $e_2$, intersect at a single point so that the surgery performed is the standard one at a single transverse intersection point.
Let $H$ be the graph obtained by replacing $e_1, e_2, e_3$ with a new edge $f_{1,3}$ which has vertices $w_1, w_3$, and is obtained travelling along $e_1$, out along $e_2$ and around the critical value $w_2$, and returning along $e_2$ and $e_3$ (See \cref{fig:dehntwistexchange}). 
\begin{figure}
	\centering
	\input{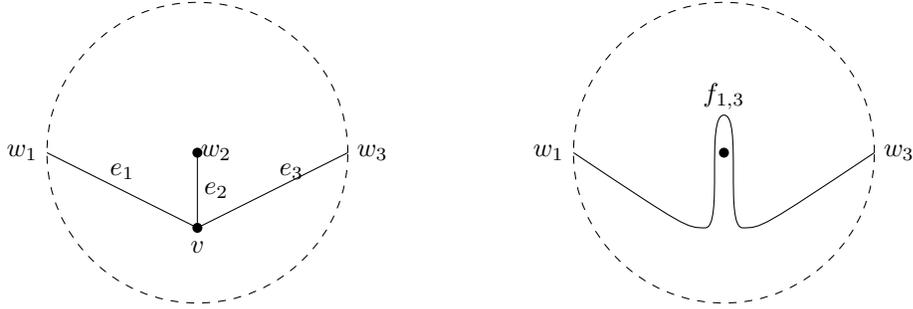}
	\caption
	{
		One can add or remove Lagrangian thimbles by exchanging them for Dehn twists.
	}
	\label{fig:dehntwistexchange}
\end{figure}
Then the graph $H= G\cup \{f_{1,3}\}\setminus\{e_i\}$ equipped with Lagrangian labelling data inherited from $G$ (with the additional label $\ell_{f_{1,3}}=\ell_{e_1}$) is again a Lagrangian glove.
We call the Lagrangian obtained via this exchanging operation  $\tau_{w_2} L_G^{\ell_e}$. 
In summary:
\begin{prop}
	The following operations produce Lagrangian isotopic Lagrangian gloves.
	\begin{itemize}
		\item
			  Any isotopy of the graph $G$ where the interior of the edges stay outside the complement of the critical values of $W$.
		\item 
				Any isotopy of the graph $G$ where an edge passes through a critical value, but the Lagrangian label of the edge is disjoint from the vanishing cycles of the critical fibers.
		\item
		      Exchanging the Lagrangian $L_G^{\ell_e}$ with  $\tau_w L_G^{\ell_e}$ at some vertex $w$.
	\end{itemize}
	\label{lemma:dehnexchange}
\end{prop}
\begin{proof}
	The first two types of modifications are clear. For the third kind of modification, see \cite[Lemma A.25]{abouzaid2018khovanov}.
\end{proof}

\subsubsection{Comparisons between tropical and Lefschetz: pants}
We now will provide a construction of a Lagrangian pair of pants in the setting of $(\CC^2\setminus \{z_1z_2=1\})$ from the perspective of the Lefschetz fibration considered in \cref{subsec:syzthimbles}:
\begin{align*}
	W: \CC^2\setminus \{z_1z_2=1\}\to       & \CC    \\
	(z_1, z_2)\mapsto & z_1z_2
\end{align*}
 See \cref{fig:lefschetzandtropical} for the correspondence between Lagrangian tori in the Lefschetz fibration and almost toric fibration. 

In this setting we build a Lagrangian glove.
We start with the Lagrangian $\ell=\RR\subset W^{-1}(-1+\epsilon)$.
For small $\epsilon<1$, we consider the loop $\gamma_\epsilon=\epsilon e^{ i \theta}-1$. The parallel transport of $\ell$ along this loop builds a Lagrangian $L_{\gamma_\epsilon}^\ell$.
The Lagrangian $L_{\gamma_\epsilon}^\ell$ only pairs against tori $F_{\epsilon, s}$, so its support in the almost toric fibration will be a line. See the blue Lagrangian as drawn in \cref{fig:glovepants}.

By exchanging a Dehn twist for an additional vertex in the glove (\cref{lemma:dehnexchange}), we can build a new Lagrangian $\tau_0L_{\gamma_\epsilon}^\ell$ (drawn in red in \cref{fig:glovepants}).
This description provides us with another construction of the Lagrangian pair of pants.
\begin{figure}
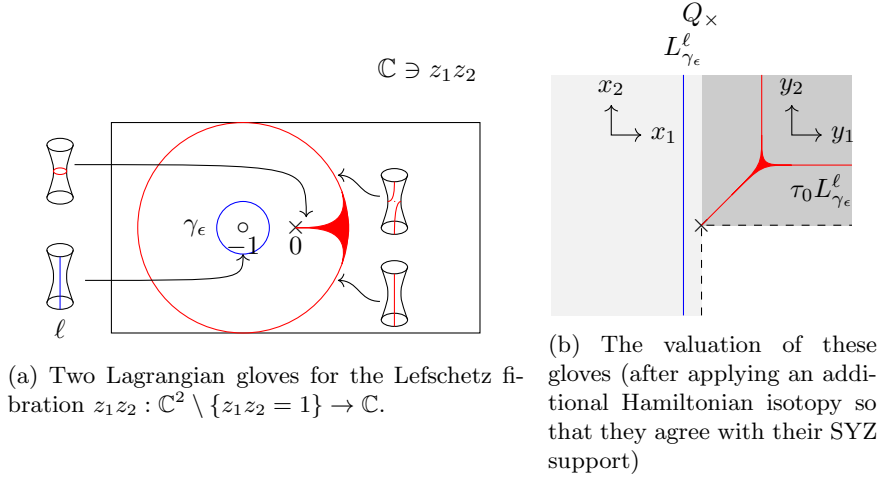

	\centering
	\begin{subfigure}{.55\linewidth}
		\centering
		\input{figures/glovelefschetz.tikz}
		\caption{Two Lagrangian gloves for the Lefschetz fibration $z_1z_2: \CC^2\setminus\{z_1z_2=1\}\to \CC$. }
	\end{subfigure}\;\;
	\begin{subfigure}[scale=.7]{.35\linewidth}
		\centering
		\input{figures/glovelefschetz2.tikz}
		\caption{The valuation of these gloves (after applying an additional Hamiltonian isotopy so that they agree with their SYZ support) }
	\end{subfigure}
		\caption{Comparing Lefschetz and tropical pictures at nodal fibers.}
	\label{fig:glovepants}
\end{figure}
These local models are compatible with the discussion from \cref{sec:almosttoric}. 
Let $Q_{\times}$ be the integral tropical manifold which is the base of $X=\CC^2\setminus\{z_1z_2=1\}$. 	
$Q_{\times}$ can be covered with two affine charts.
Call the charts 
\begin{align*}
Q_0=&\{(x_1, x_2)\}\setminus\{(x, x)\;|\; x>0\}\\
Q_1=&\{(y_1, y_2)\}\setminus \{(y,y)\;|\; y<0\}.
\end{align*}
The charts are glued with the change of coordinates 
\[
	(y_1, y_2)=\left\{\begin{array}{cc}
	(x_1, x_2 )& x_2> x_1\\
	(2x_1-x_2, x_1)& x_2<x_1
	\end{array}
	\right.
\]
We now consider two tropical curves inside of $Q_{\times}$. 
The first is an affine line, which is given by the critical locus of a tropical polynomial defined over the $Q_0$ chart 
\[
	\phi_0(x_1, x_2)=1\oplus x_1.
\]
The second tropical curve we consider is a pair of pants with a capping thimble (as described in \cref{sec:almosttoric},) given by the critical locus of a tropical polynomial defined over the $Q_1$ chart, 
\[
	\phi_1(y_1, y_2)=y_1\oplus y_2\oplus 1.
\]
From \cref{lemma:dehnexchange}, we get the following corollary:
\begin{corollary}[Nodal Trade for Tropical Lagrangians]
	Consider the tropical curves $V(\phi_0)$ and $V(\phi_1)$ inside of $Q_\times$. 
	The Lagrangians $L(\phi_0)$ and $L(\phi_1)$ are Lagrangian isotopic in $\CC^2\setminus \{z_1z_2=1\}$. 
	\label{cor:tropicalexchange}
\end{corollary}
This corollary allows us to manipulate tropical Lagrangians by  manipulating the tropical diagrams in the affine tropical manifold instead. 

\begin{example}
	Consider the Lefschetz fibration with fiber $\C^*$ given by the smoothed $A_n$ singularity  as in \cref{fig:ansingularity}. 
	We construct the Lagrangian glove where we parallel transport the real arc $\ell=\RR\subset \CC^*$ around the loop of the glove. 
	The monodromy of the symplectic connection from travelling around the large circle corresponds to $n$ twists of the same vanishing cycles. By attaching $n$ vanishing cycles to this arc, we get a Lagrangian glove. In the ``moment map'' picture, all of the singularities lie on the same eigenray, and we get the tropical Lagrangian which is a $n+2$ punctured sphere with $n$ of the punctures filled in with thimbles. Though it appears that the $n$ thimbles of the Lagrangian coincide with each other in the ``moment map'' picture, they differ by some amount of phase in the fiber direction, which is easily seen in the Lefschetz fibration.
\begin{figure}
	\centering
	\input{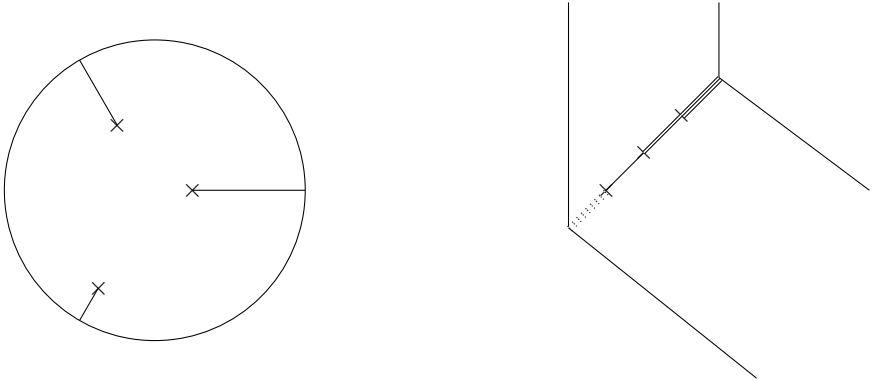}
	\caption{
		The resolved $A_3$ singularity, a Lagrangian glove, and its associated tropical curve.
	}
	\label{fig:ansingularity}
\end{figure}
\end{example}

\begin{corollary}
	The Lagrangians from \cref{exam:innertropical,exam:outertropical} are Lagrangian isotopic. 
	\label{cor:innerouterisotopic}
\end{corollary}

\section{Lagrangian tori in toric del-Pezzos}
\label{sec:tdp}
We now introduce a monotone Lagrangian torus which exists in a toric del-Pezzo. 
We show that in the setting of $\CP^2$ this Lagrangian $L_{T^2}$ is isotopic to $F_q$, a fiber of the moment map. 
Finally, we speculate on homological mirror symmetry for $L_{T^2}\subset \CP^2\setminus E$, where this Lagrangian is no longer isotopic to $F_q$. 
We exhibit a symplectomorphism $g: \CP^2\setminus E\to \CP^2\setminus E$ expected to be mirror to fiberwise Fourier-Mukai transform on the mirror. 
\subsection{Examples from toric del-Pezzos}
\label{subsub:examplesfromdelpezzos}
Monotone Lagrangian tori and Lagrangian seeds in del-Pezzo surfaces have been studied in \cite{vianna2017infinitely,pascaleff2020wall}.
Let $X_\Sigma$ be a toric del-Pezzo.
A choice of monotone symplectic structure on $X_\Sigma$ gives a monotone Lagrangian torus $F_\Sigma$ at the barycenter of the moment polytope has a Lagrangian seed structure $\{D_{i,\Sigma}\}$ given by the Lagrangian thimbles extending from the corners of the moment polytope. The Lagrangian thimbles and corresponding dimers are drawn in   \crefrange{fig:fano1}{fig:fano5}.
\begin{figure}
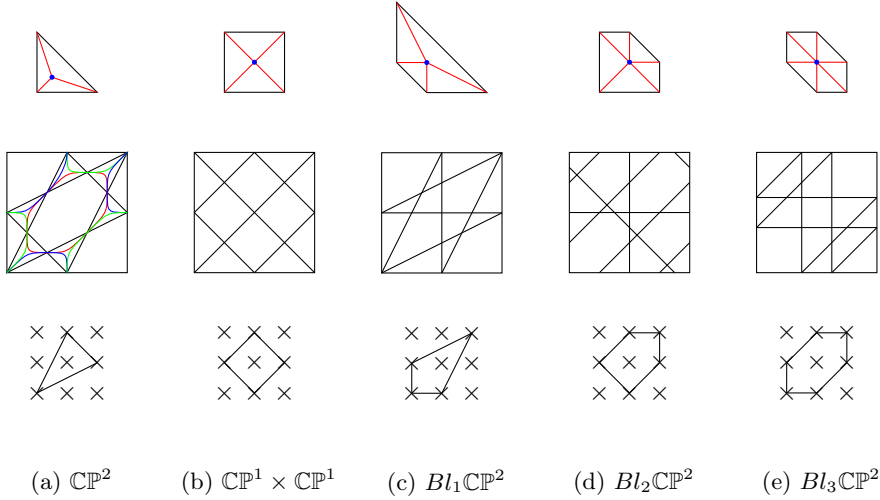

	\begin{subfigure}{.19\linewidth}
		\centering
		\input{figures/fano1.tikz}
		\caption{$\CP^2$}
		\label{fig:fano1}
	\end{subfigure}
	\begin{subfigure}{.19\linewidth}
		\centering
		\input{figures/fano2.tikz}
		\caption{$\CP^1\times \CP^1$}
		\label{fig:fano2}
	\end{subfigure}
	\begin{subfigure}{.19\linewidth}
		\centering
		\input{figures/fano3.tikz}
		\caption{$Bl_1\CP^2$}
		\label{fig:fano3}
	\end{subfigure}
	\begin{subfigure}{.19\linewidth}
		\centering
		\input{figures/fano4.tikz}
		\caption{$Bl_2\CP^2$}
		\label{fig:fano4}
	\end{subfigure}
	\begin{subfigure}{.19\linewidth}
		\centering
		\input{figures/fano5.tikz}
		\caption{$Bl_3\CP^2$}
		\label{fig:fano5}
	\end{subfigure}
	\caption[Toric del Pezzos, Seeds, and dimer Lagrangians]{\textbf{Top}: Lagrangian seeds in toric del Pezzo surfaces. The antisurgery disks are drawn in red. \textbf{Middle}: The corresponding dual dimer models associated the Lagrangian seeds. In the first example of $\CP^2$, we additionally draw the classes of the cycles $\partial D_{f_{i, \Sigma}}\subset F^*_\Sigma$.  \textbf{Bottom:} Cycle classes of the zigzag diagram, corresponding to mutation directions. }
	
	\label{fig:fanotori}
\end{figure}
In these 5 examples, the dimer Lagrangian $F_\Sigma^*$ constructed from the data of $(F_\Sigma, D_{i, \Sigma})$ again has the topology of a torus.
This can be checked from the computation of the Euler characteristic of the dual Lagrangian, 
\[
	\chi(L^*)=|V(G)|-|E(G)|+|\Sigma|,
\]
where $|\Sigma|$ is the number of antisurgery disks of $F_\Sigma$. 

One method of distinguishing Lagrangians is to compute their open Gromov-Witten potentials.
In the case of toric Fanos, it was proven in \cite{tonkonog2018string}  that all Lagrangian tori have the potentials given by one of those in \cite{vianna2017infinitely}. 
A computation shows that the Lagrangians $F_\Sigma$ and $F_\Sigma^*$ have the same mutation configuration. 
We will use this information to later match their Landau-Ginzburg potentials.
\begin{claim}
	Let $X_\Sigma$ be a toric Fano, $F_\Sigma$ the standard monotone Clifford torus in $X_\Sigma$, and $F_\Sigma^*$ be the dual torus constructed using the  Lagrangian seed structure on $F_\Sigma$.
	There is a set of coordinates for $H_1(F_\Sigma^*)$ and $H_1(F_\Sigma)$ so that the mutation directions determined by their Lagrangian seed structures are the same.
	\label{claim:mutationdirections}
\end{claim}
\begin{proof}
	This is done by an explicit computation of the homology classes of the disk boundaries in $F_\Sigma^*$.
\end{proof}

As a corollary, the wall and chamber structure on the moduli space of Lagrangians $F_\Sigma$ obtained by mutations may be replicated in a similar fashion on the moduli space of the Lagrangians $F_\Sigma^*$.
\begin{corollary}
	In the setting of toric Fanos, the Landau-Ginzburg potential of $F_\Sigma$ is the same as $F_\Sigma^*$. \label{cor:superpotential}
	\label{subsub:torusmutation}
\end{corollary}
In both \cref{fig:fano1,fig:fano2} we may mutate the diagram to give us a dimer model with two polygons, which is the balanced tropical Lagrangian for some tropical polynomial.
As a result, the Lagrangians \cref{fig:fano1,fig:fano2} are Lagrangian isotopic to tropical Lagrangians constructed in \cref{sec:almosttoric}.
In the case of \cref{fig:fano1}, the dimer Lagrangian  $F^*_{\Sigma_{\CP^2}}$ and Lagrangian and $L_{inner}$ from \cref{fig:tpIII} are related by Lagrangian mutation and isotopy; the local model of mutation is drawn in \cref{fig:toricdimermutation}.
This is perhaps easier to see geometrically by working in the reverse direction. First apply isotopy by contracting the tropical genus of $L_{inner}$ (\cref{fig:smoothtorus}).
The underlying tropical curve becomes nonsmooth, and this isotopy sweeps out some symplectic flux while collapsing a cycle of $L_{inner}$.
The resulting Lagrangian is an immersed Lagrangian (whose model at the vertex is  \cref{fig:mutatedhexagon}). 
Surgering this immersed point in the other direction recovers $F^*_{\Sigma_{\CP^2}}$ (\cref{fig:mutationhexagons}).

\begin{remark}
	 $L_{inner}$ is not the mutation of $F^*_{\Sigma_{\CP^2}}$, as mutation is the composition of antisurgery and surgery with equal neck sizes.
	 $L_{inner}$ is obtained from $F^*_{\Sigma_{\CP^2}}$ by antisurgery followed by surgery, but with some symplectic flux swept out due to non-equal neck sizes chosen. 
	In particular, $F^*_{\Sigma_{\CP^2}}$ is monotone, as is $\mu_D F^*_{\Sigma_{\CP^2}}$. However,  $L_{inner}$ is not. 
	\label{rem:linnermonotone}
\end{remark}

It is unclear how much of this story extends beyond the toric case.
\begin{question}
	Is there a relation between  $(L, D_i)$ and $(L^*, D^*_f)$ that can be stated in the language of mirror symmetry?
\end{question}

We conclude our discussion with a collection of observations for mirror symmetry of $\CP^2\setminus E$ and the elliptic surface $\check X_{9111}$.
Here, $\check X_{9111}$ is the extremal elliptic surface in the notation of \cite{miranda1989basic}. 
This elliptic surface has a Lefschetz fibration $W_{9111}:\check X_{9111}\to \CP^1$ with 3 singular fibers of type $I_1$, and one singular fiber of type $I_9$.
An $I_k$ fiber is the degenerate elliptic fiber which is a $k$-chain of $\CP^1$s. 
We can present this elliptic surface \cite[Table Two]{artebani2016cox} as the blowup of a pencil of cubics on $\CP^2$,
\[
	(z_1^2z_2+z_2^2z_3+z_3^2z_1)+t\cdot(z_1z_2z_3)=0.
\]
From this pencil, we get a map $\check \pi_{bl}: \check X_{9111}\to \CP^2$, which has nine exceptional divisors. 
Three of the exceptional divisors correspond to the base points of the pencil giving us three sections of the fibration $\check W_{9111}:\check X_{9111}\to \CP^1$. 
We study homological mirror symmetry with the $A$-model on $\CP^2$.
Of principle interest will be the Lagrangian discussed in \cref{fig:tpIII,}, which we will call $L_{inner}\subset \CP^2$. The Lagrangian discussed in \cref{fig:tpII} will be called $L_{outer}\subset \CP^2$.

In \cref{subsec:comparingtori}, we use methods from \cref{subsec:tropicalpantslefschetz} to compare the Lagrangian $L_{T^2}$ to a fiber $F_q\subset \CP^2$ of the moment map. Finally, we make a homological mirror symmetry statement for $L_{T^2}$ and the fibers of the elliptic surface $\check X_{9111}$ in \cref{subsec:amodeloncp2}.

\subsection{Tropical Lagrangian Tori in \texorpdfstring{$\CP^2$.}{CP2}}
\label{subsec:comparingtori}
We now apply the tools from Lefschetz fibrations to give us a better understanding of the tropical Lagrangians in $\CP^2$ from \cref{fig:fano1}.
\begin{prop}
	The Lagrangian $L_{inner}$ drawn in  \cref{fig:fano1} is Lagrangian isotopic to the moment map fiber $F_p$ of $\CP^2$.
	\label{prop:tropicalischekanov}
\end{prop}
This relation is already somewhat expected. 
\cite{vianna2014infinitely} provides an infinite collection of monotone Lagrangian tori which are constructed by mutating the product monotone tori along different mutation disks. It is conjectured that these are all of the monotone tori in $\CP^2$. 
From \cref{subsub:torusmutation} we know that the Lagrangian $L_{T^2}$ has the same Lagrangian mutation seed structure as $T^2_{prod, mon}$, so if this conjecture on the classification of Lagrangian tori in $\CP^2$ holds, these two tori must be Hamiltonian isotopic.
\begin{proof}
	The outline is as follows: we first use the isotopy provided by \cref{cor:innerouterisotopic} between $L_{inner}$ and $L_{outer}$.
	We then compare the Lagrangians $L_{outer}$ to a Lagrangian glove for a Lefschetz fibration.
	This Lefschetz fibration is constructed  from a pencil of elliptic curves chosen for a large amount of symmetry. 
	Finally, we compare $F_p$ to the Lagrangian constructed via a Lefschetz fibration.
	 The Lagrangians $F_p$ and $L_{outer}$ are matched via an automorphism of the pencil of elliptic curves.

	We first will talk about the geometry of the pencil and the automorphism we consider. The Hesse pencil of elliptic curves is the one parameter family described by
	\[
		 (z_1^3+z_2^3+z_3^3)+t \cdot (z_1z_2z_3)=0
	\]
	which has four degenerate $I_3$ fibers at symmetric points $p_1, p_2, p_3, p_4\in  \CP^1$.
	Let $E_{12}\subset \CP^2$ be the member of the pencil  whose projection to the parameter space $\CP^1$  is the midpoint $p_{12}$ between  $p_1$ and $p_2$.
	The generic fiber of the projection $W_{3333}:\CP^2\setminus E_{12}\to \CC$ is a 9-punctured torus. From each $I_3$ fiber we have three vanishing cycles.
	After picking paths from these degenerate fibers to a fixed point $p\in \CC$, we can match the vanishing cycles to the cycles in $E_p$ as drawn in \cref{fig:x3333}.
	\begin{remark}
		A small digression, useful for geometric intuition but otherwise unrelated to this discussion, concerning the apparent lack of symmetry in the vanishing cycles of $W_{3333}$.
		One might expect that the configuration of vanishing cycles which appear in \cref{fig:x3333} to be entirely symmetric.
		While the Hesse pencil has symmetry group which acts transitively on the $I_3$ fibers, to construct the vanishing cycles one must pick a base point $p$ and a basis  of paths from $E_p$ to the critical fibers of the Hesse configuration, which breaks this symmetry. 
		Each path from a point $p$ to one of the four critical values $p_i$ gives us 3 parallel vanishing cycles. 
		The 4 critical fibers of the Hesse configuration lie at the corners of an inscribed tetrahedron on $\CP^1$. 
		By choosing $p=p_{123}$ to be the center of a face spanned by three of these critical values, 3 paths (say, $\gamma_1, \gamma_2, \gamma_3$) from $p$ to the critical values are completely symmetric.
		From such a choice, we obtain vanishing cycles $\ell_1^j, \ell_2^j, \ell_3^j$, where $j\in \{1, 2, 3\}$.
		The homology classes (and in fact, honest vanishing cycles) 
		\begin{align*}
			|\ell_1^j|=\langle 0 ,1\rangle&& |\ell_2^j|=\langle 1, 0\rangle&& |\ell_3^j|=\langle 1, 1\rangle 
		\end{align*}
		are indistinguishable after action of $SL(2, \ZZ)$, reflecting the overall symmetry of both the $X_{3333}$ configuration and the symmetry of the paths. The action of $SL(2, \ZZ)$ which interchanges these cycles also permutes the 9 points of $E_{p_{123}}$ which are the base points of this fibration.
		
		However, the introduction of the last path from the fourth critical fiber to $p_{123}$ breaks this symmetry. At best, this path can be chosen so that there remains one symmetry, which exchanges $\ell_1$ and $\ell_2$.
		In this setup, the vanishing cycles $\ell_4^i$ lies in the class $\langle 1, -1 \rangle$.
		Correspondingly, the class $\langle 1, -1\rangle$ distinguishes the class $\ell_3$ from the other classes by intersection number.
		
	\end{remark}

	This pencil is sometimes called the \emph{anticanonical pencil of $\CP^2$. }
	\begin{figure}
		\centering
		\input{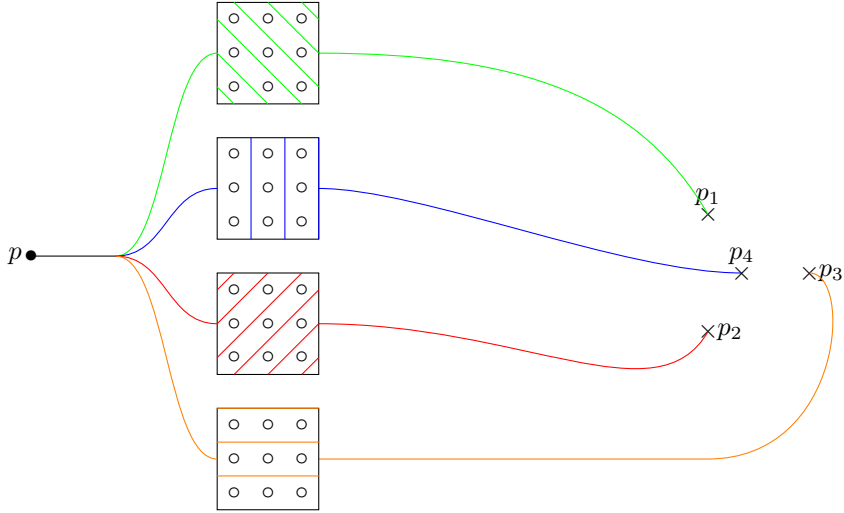}
		\caption
		{
			A basis for the vanishing cycles for $X_{3333}$ given in \cite{seidel2016fukaya}.
		}
		\label{fig:x3333}
	\end{figure}
	The automorphism group of the Hesse pencil is called the Hessian Group \cite{jordan1877memoire}. This group acts on $\CP^1$ by permuting the critical values by even permutations.
	Consider a pencil automorphism $g: \CP^2\to \CP^2$ which acts on the 4 critical values via the permutation $(p_1p_2)(p_3p_4)$. The point $p_{12}$ is fixed under this action, therefore $g(E_{12})=E_{12}$. 
	While the fiber $E_{12}$ is mapped to itself, the map is a non-trivial automorphism of the fiber, swapping the vanishing cycles for $p_1$ and $p_2$: 
	\begin{align*}
		g(\ell_1)=&\ell_2\\
		g(\ell_2)=&\ell_1.
	\end{align*}

	We can use the Lefschetz fibration to associate to each cycle $\ell$ in $E_{12}$ a Lagrangian in $\CP^2$ by taking the Hamiltonian suspension cobordism of $\ell$ in a small circle $p_{12}+\epsilon e^{i\theta}$ around the point $p_{12}$ in the base of the Lefschetz fibration.
	Call the Lagrangian torus constructed this way $T_{\epsilon, \ell}$.
	The automorphism of the pencil $g: \CP^2 \to \CP^2$ interchanges the Lagrangians $T_{\epsilon, \ell_1}$ and $T_{\epsilon , \ell_2}$

	The standard moment map $\val_{dz}:\CP^2\to Q_{\CP^2,dz}$ can be chosen so that one of the $I_3$ fibers of the Hesse configuration projects to the boundary of the Delzant polygon $Q_{\CP^2}.$
	We choose the moment map so that $\val^{-1}_{dz}(\partial Q_{\CP^2,dz})=E_1$, the $I_3$ fiber lying above the point $p_1$.
	When one performs a nodal trade exchanging the corners of the moment map for interior critical fibers, we obtain a new toric base diagram, $Q_{\CP^2}$. The boundary of the base of the almost toric fibration $\val:  \CP^2\to Q_{\CP^2}$  corresponds to a smooth symplectic torus.
	We arrange that  
	\[
		\val^{-1}(\partial Q_{\CP^2})=E_{12}\subset \CP^2.
	\]
	By comparison to the standard moment map, one sees that the cycle $\ell_1\subset E_{12}$ projects to a point in the boundary of the moment map, while the cycle $\ell_2\subset E_{12}$ projects to the whole boundary cycle. 
	This gives us an understanding of the valuation projections of Lagrangian $T_{\epsilon ,\ell_1}$ and $T_{\epsilon ,\ell_2}$. 
	$T_{\epsilon,\ell_1}$ has valuation projection which roughly looks like a point, and $T_{\epsilon, \ell_2}$ has valuation projection which is a cycle that travels close to the boundary of $Q_{\CP^2}$. As a result we have Hamiltonian isotopies identifying the Lagrangians
	\begin{align*}
		T_{\epsilon, \ell_1}\sim & F_p       \\
		T_{\epsilon, \ell_2}\sim & L_{outer}.
	\end{align*}
	See \cref{fig:x3333tropical}, where $L_{outer}$ is drawn in red, and $F_p$ is drawn in blue.
	\begin{figure}
		\centering
		\input{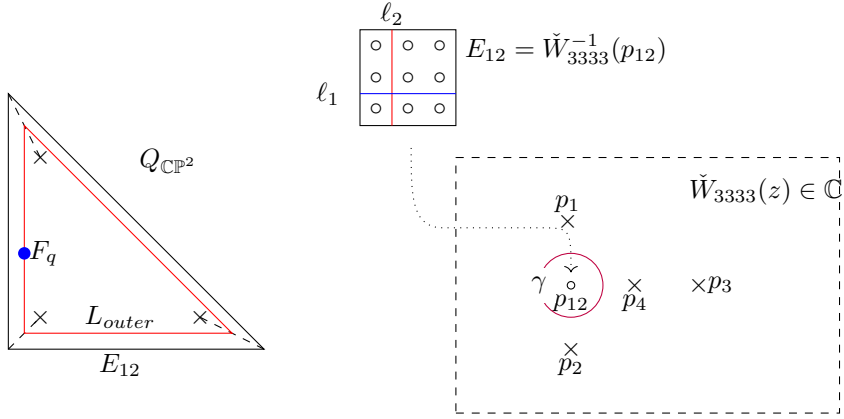}
		\caption{Relating tropical Lagrangians to thimbles}
		\label{fig:x3333tropical}
	\end{figure}

	We conclude $g(L_{outer})\sim F_p$.
	As the projective linear group is connected, the morphism $g$ is symplectically isotopic to the identity, and since $H^1(\CP^2)$ is trivial, all symplectic isotopies are Hamiltonian isotopies. 
	Therefore the Lagrangians $L_{outer}$ and $F_p$ are Hamiltonian isotopic.

	By \cref{cor:tropicalexchange}, the Lagrangians $L_{inner}$ and $F_p$ are Lagrangian isotopic.
\end{proof}

This shows that $L_{inner}$ is obtained from a Lagrangian that we've seen before, but presented from a very different perspective.
By taking a Lagrangian isotopy, $L_{outer}$ can be moved to $L_{inner}$. 
We obtain the following relationships between Lagrangian submanifolds. 
Here, the equalities are taken up to Hamiltonian isotopy, and the dashed lines are Lagrangians which we expect to be Hamiltonian isotopic. 
See also \cref{fig:toricdimermutation,rem:linnermonotone}.
\[
	\begin{tikzcd}
		L_{T^2} \arrow[dashed, equals]{d} \arrow[dash]{r}{\text{mutation}}& \mu_{D_f}L_{T^2} \arrow{r}{\text{Lag. Isotopy}}  \arrow[dashed, equals]{d} &  L_{inner} \arrow[dash]{r}{\text{Lag. Isotopy}} & L_{outer} \arrow[equals]{d}\\
		T^2_{prod, mon} \arrow[dash]{r}{\text{mutation}} & T^2_{chek, mon} \arrow[dash]{rr}{\text{Lag. Isotopy}}& & F_p.
	\end{tikzcd}.
\]
 
These tori are isomorphic objects of the Fukaya category, but this is a consequence of $\Fuk(\CP^2)$ having so few objects.
\subsection{\texorpdfstring{$A$}{A}-Model on \texorpdfstring{$\CP^2\setminus E$}{CP2\backslash E}. }
\label{subsec:amodeloncp2}
We now study the map $g:\CP^2\to\CP^2$ given by the automorphism of the Hesse configuration, and its action on the Fukaya category. 
Informally, the Fukaya category is an $A_\infty$ category associated to a symplectic manifold $X$.
The objects of the category are monotone or unobstructed Lagrangian submanifolds, and the morphisms between two Lagrangian submanifolds is their Lagrangian intersection Floer cohomology $\CF(L_0, L_1)$.
The product $m^2: \CF(L_1, L_2)\tensor \CF(L_0, L_1)\to \CF(L_0, L_2)$ is given by counting holomorphic triangles with boundary on $L_0, L_1, L_2$ and strip-like ends limiting to intersection between $L_1\cap L_2, L_0 \cap L_1,$ and $L_0\cap L_2$.
Higher products similarly count holomorphic $k+1$-gons with boundary on $L_0, \ldots, L_k$ and strip-like ends limiting to the intersections between $L_{i}\cap L_{i+1}$.

As there are few monotone Lagrangians in $\CP^2$, the category $\Fuk(\CP^2)$ does not contain many objects, so the automorphism of the Fukaya category induced by $g$ is not so interesting.
By removing an anticanonical divisor $E=E_{12}$ we obtain a much larger category.
For example, the Lagrangians $L_{outer}$ and $F_q$ are no longer Hamiltonian isotopic in $\CP^2\setminus E$.
\begin{claim}
	$L_{outer}$ and $F_q$ are not isomorphic objects of $\Fuk(\CP^2\setminus E)$
\end{claim}
\begin{proof}
	The symplectic manifold $\CP^2\setminus E$ contains a Lagrangian thimble $\tau_1$ which is constructed from the singular fiber of the almost toric fibration and extends out towards the removed curve $E$ (see \cref{fig:tpIV}).
	This thimble $\tau_1$ intersects $L_{outer}$ at a single point, and therefore $\CF(L_{outer}, \tau_1)$ is nontrivial.
	However, $\tau_1$ is disjoint from the fiber $F_q$, so $\CF(F_q, \tau_1)$ is trivial.
	As a result, $F_q$ and $L_{outer}$ are not isomorphic objects of the Fukaya category.\footnote{In fact, the same argument shows that $L_{outer}$ and $F_q$ are not topologically isotopic.}
\end{proof}
Since $E_{12}$ was fixed by the symplectomorphism $g: \CP^2\to \CP^2$, the restriction to the complement $g: \CP^2\setminus E_{12}\to \CP^2\setminus E_{12}$ is still defined.
\begin{corollary}
	The automorphism of the Fukaya category induced by the symplectomorphism $g$
	\[
		g^*: \Fuk(\CP^2\setminus E)\to \Fuk(\CP^2\setminus E)
	\]
	acts nontrivially on objects.
	\label{cor:nontrivial}
\end{corollary}
This section of the paper is a series of observations and conjectures outlining homological mirror symmetry with the $A$-model on $\CP^2\setminus E$, and $B$-model on $\check X_{9111}$ which hope to shed light on the following conjecture. 
\begin{conjecture}
	The symplectomorphism $g: \CP^2\to \CP^2$ is mirror to fiberwise Fourier Mukai transform on the elliptic surface $\check X_{9111}$ which interchange the points of $\check X_{9111}$ with line bundles supported on the fibers of the elliptic fibration.
\end{conjecture}

\subsubsection{Homological mirror symmetry for $\CP^2\setminus E$}
To that end, we study $L_{T^2}\subset \CP^2\setminus E$.
\paragraph{An intermediate Blowup and Base Diagrams for $\check X_{9111}$:} 
We will  begin with a description of the elliptic surface $\check X_{9111}$ as an iterated blow up of $\CP^2$ along the base points of an elliptic pencil following \cite{auroux2006mirror}. Consider the pencil
\[
	(z_1^2z_2+z_1z_2^2+z_3^3)+tz_1z_2z_3=0.
\]
This elliptic fibration has 3 base points of degree 4, 4, and 1.
\begin{figure}
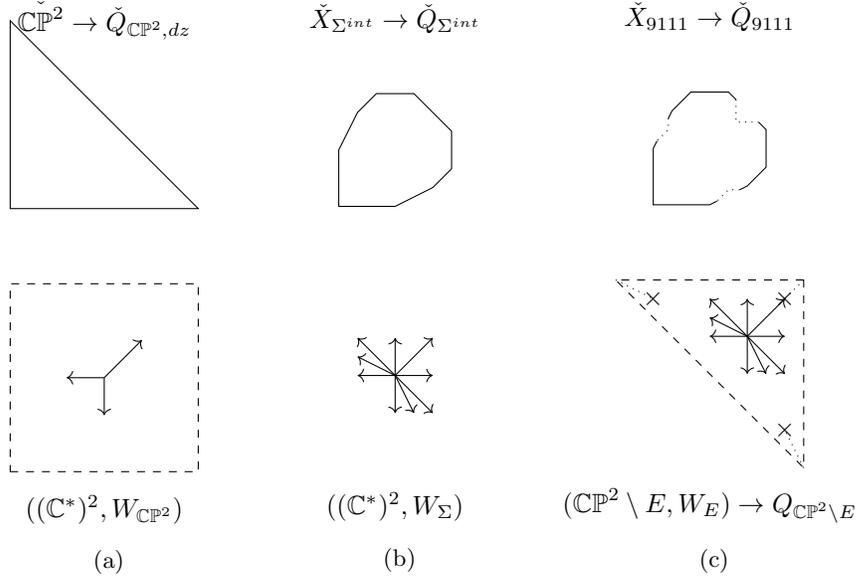

	\centering
	\begin{subfigure}{.3\linewidth}
		\centering
		\input{figures/mirrorsymmetrycp2.tikz}
		\caption{}
		\label{fig:mirrorsymmetrycp2}
	\end{subfigure}
	\begin{subfigure}{.3\linewidth}
		\centering
		\input{figures/mirrorsymmetryxint.tikz}
		\caption{}
		\label{fig:mirrorsymmetryxint}
	\end{subfigure}
	\begin{subfigure}{.35\linewidth}
		\centering
		\input{figures/mirrorsymmetryx9111.tikz}
		\caption{}
		\label{fig:mirrorsymmetryx9111}
	\end{subfigure}
	\caption[Blow up $\check X_{9111}\to\CP^2$ and mirrors]{\textbf{Top}: obtaining $\check X_{9111}$ as a toric base diagram by first blowing up $\CP^2$ 6 times, then  blowing up 3 more times. \textbf{Bottom}: Admissibility conditions for the $A$-model mirrors. }
	\label{fig:x9111}
\end{figure}
We can arrange for 6 of the blowups  (3 on the two base points of degree 4) to be toric.
We therefore obtain an intermediate step between $\check \CP^2$ and $\check X_{9111}$ which is the toric symplectic manifold $\check X_{\Sigma^{int }}$.
The toric diagram $ Q_{\Sigma^{int }}$ is the Delzant polytope with 9 edges.
The remaining 3 blowups introduce nodal fibers in the toric base diagram $\check Q_{9111}$ for $\check X_{9111}$ which has 9 edges and 3 nodal fibers.
The 9 edges of the toric base correspond to the nine $\CP^1$'s making the $I_9$ fiber of the fibration.
The eigenray at each cut in the diagram is parallel to the boundary curves.
See \cref{fig:x9111} for the base diagrams of these different blowups.
Let  $\check W_{9111}: \check X_{9111}\to \CP^2 \to \CP^1$ be the composition of blowdown and projection to the $t$ parameter of the pencil.
\paragraph{$B$-model of $\check X_{9111}$:} 
Let $\check \pi: \check X_{9111}\to \check X_{\Sigma^{int }}$ be the projection of the last three blowups.
By \cite{bondal1995semiorthogonal} have a semiorthogonal decomposition of the category of the blowup as
\[
	D^b\Coh(\check X_{9111})= \langle \check \pi^{-1} D^b\Coh(\check X_{\Sigma^{int }}), \mathcal O_{D_1}, \mathcal O_{D_2}, \mathcal O_{D_3 } \rangle,
\]
where the $D_i$ are the exceptional divisors of the last three blowups. 
For sheaves $\mathcal O_H\in D^b\Coh(\check X_{\Sigma^{int }})$ with support on a hypersurface $H$, this semiorthogonal decomposition states that there is a corresponding sheaf in $\check X_{9111}$ whose support is on the total transform of $H$. 
Should $H$ avoid the points of the blow-up, the total transform will have the same valuation projection as $H$.
Should $H$ contain the point of the blowup, the total transform includes the exceptional divisor of the blow-up.
A fiber of the elliptic surface $W^{-1}_{9111}(p)$ is the proper transform of $E_p\subset X_{\Sigma^{int }}$ , a member of the pencil after 6 blowups.
On sheaves, we have an exact sequence:
\begin{equation}
	 \left( \bigoplus_{i=1}^3\mathcal O_{D_i}\right) \to \pi^{-1}(\mathcal O_{E_p}) \to \mathcal O_{ W_{9111}^{-1}(p)}.
	\label{eq:stricttransform}
\end{equation}

We will now set up some background necessary to state a similar story for the $A$-model, summarized in \cref{assum:x9111hms}.

\paragraph{Base for $\CP^2\setminus E$:} 
Running the machinery of \cite{gross2003affine} on  $\check Q_{9111}$, the SYZ base for $\check X_{9111}$ in the complement of a smooth anti-canonical divisor, will yield the SYZ base $Q_{\CP^2\setminus E}$ for $\CP^2\setminus E$.
The base diagram $Q_{\CP^2\setminus E}$ can also be constructed by first constructing the mirror to the space $\check X_{\Sigma^{int }}$.
As $\check X_{\Sigma^{int }}$ is a toric variety, the mirror space is a Landau Ginzburg model $(X_{\Sigma^{int }}, W_{\Sigma^{int }})=((\CC^*)^2, W_{\Sigma^{int }})$, where the superpotential $W_{\Sigma^{int }}$ yields a monomial admissibility condition (in the sense of \cite{hanlon2019monodromy}) $\Delta_{\Sigma^{int }}$ on $Q_{\Sigma^{int }}=\RR^2$.
\begin{assumption}[Monomial Admissible Blow-up]
	There is notion of monomial admissibility condition for $\CP^2\setminus E$. This monomial admissibility condition is constructed from the data of the monomial admissibility condition  $((\CC^*)^2, W_{\Sigma^{int }})$.
\end{assumption}

We now provide some motivation for this assumption. 
\begin{figure}
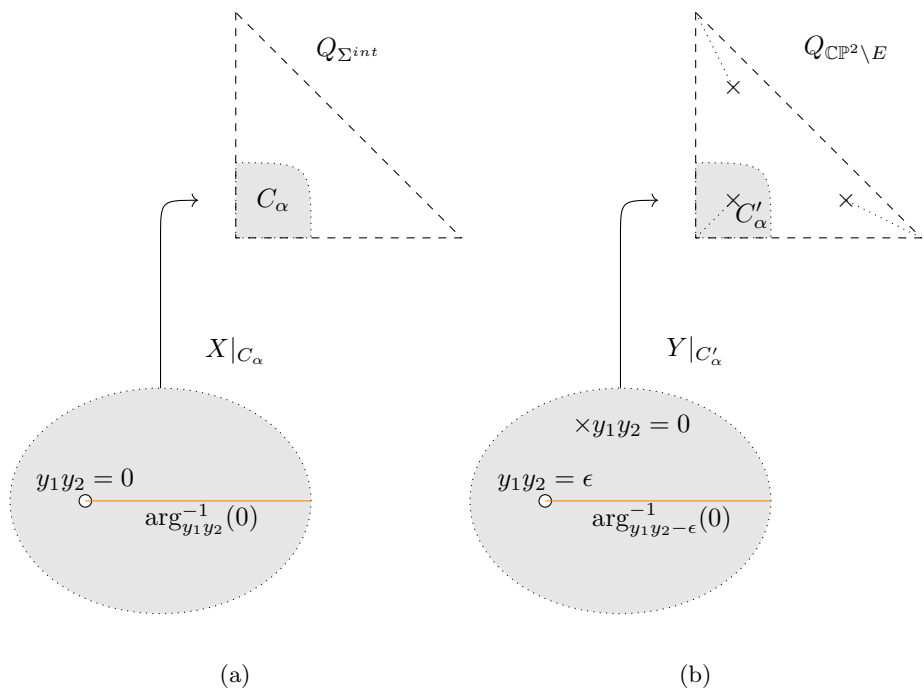

	\centering
	\begin{subfigure}{.48\linewidth}
		\centering
		\input{figures/admissibilitynearcorner1.tikz}
		\caption{}
	\end{subfigure}
	\begin{subfigure}{.48\linewidth}
		\centering
		\input{figures/admissibilitynearcorner2.tikz}
		\caption{}
	\end{subfigure}
	\caption[Relating Monomial and Lefschetz models for blowup.]
	{
		Relating Lagrangians and Admissibility conditions between $(\CC^*)^2$ and $\CP^2\setminus E$ with local Lefschetz models near corners.
	}
	\label{fig:introducingcuts}
\end{figure}
Recall, a monomial admissibility condition assigns to each monomial $c_\alpha z^\alpha$ a closed set $C_\alpha\subset Q$ on which $\arg_{c_\alpha z^{\alpha}}(L|_{C_\alpha})=0$.
For a set $C_\alpha$, denote by $X_{\Sigma^{int}}|_{C_\alpha}$  the portion of the SYZ valuation with valuation lying inside of $C_\alpha$.
Let $L|_{C_\alpha}$ be the restriction of a Lagrangian $L$ to the preimage $\val^{-1}(C_\alpha)$.
A Lagrangian is monomial admissible if for every $\alpha$,  $L|_{C_\alpha}\subset \arg_{c_{\alpha} z^{\alpha}}^{-1}(0)$.
The projection $\arg_{c_\alpha z^\alpha}^{-1}(0)\to C_\alpha$ is an $S^1$ subbundle of the SYZ fibration $X|_{C_{\alpha}}\to C_\alpha$.

To obtain $Q_{\CP^2\setminus E}$ from $Q_{\Sigma^{int }}$, we add in three cuts mirror to the three blowups. 
These three cuts are added by replacing the regions $C_{z_1z_2}, C_{z_1z_2^{-2}}$ and $C_{z_1^{-2}z_2}$ with affine charts $C_{z_1z_2}', C_{z_1z_2^{-2}}'$ and $C_{z_1^{-2}z_2}'$ each containing a nodal fiber.
The charts $C_\alpha$  can be locally modelled on $\CC^2\setminus \{y_1y_2=0\}$ with monomial admissibility condition $(y_1y_2)^{-1}$.
We replace these with charts containing a nodal fiber modeled on $Y:=\CC^2\setminus \{y_1y_2=\epsilon\}$ and admissibility condition controlled by the monomial $(y_1y_2-\epsilon)^{-1}$.
The valuation map $Y|_{C'_\alpha}\to C_{\alpha}'$ is an almost toric fibration.
We still have an $S^1$ subbundle  $\arg_{y_1y_2-\epsilon}^{-1}(0)\subset Y|_{C'_\alpha}$ of the SYZ fibration  $Y|_{C'_\alpha}\to C_{\alpha}'$ whenever $\epsilon$ is not negative real.
This $S^1$-subbundle and the monomial $(y_1y_2-\epsilon)^{-1}$ should be used to construct a monomial admissibility condition on $\CP^2\setminus E$. See \cref{fig:introducingcuts,fig:lagrangiansincp2e}.
From mirror symmetry considerations we expect the function $W_E: \CP^2\setminus E\to \CC$ to decompose as the sum of 9 monomials, corresponding to a specific basis of $H^0(\CP^2, \mathcal O(E))$.
This corresponds to a decomposition of the open Gromov-Witten potential on the mirror space, which counts the number of Maslov 2-holomorphic disks with boundary on a SYZ fiber $(F_q, \nabla)\subset \check X_{9111}$.
For $F_q\subset \check X_{9111}\setminus I_9$, each disk contributing to the open Gromov-Witten potential must intersect the $I_9$ anticanonical divisor; the potential can be split into monomial terms by restricting to those disks with pass through a specific $\CP^1\subset I_9$ component. 

In terms of the almost toric base diagrams, this compatibility can be stated as a matching between the eigendirection of the introduced cuts and the ray of the fan corresponding to the controlling monomial over the region including the cut.

\paragraph{$A$-model on $\CP^2\setminus E$}
We now conjecture the existence of a mirror to the inverse-image functor on the $B$-model.
Lagrangian submanifolds which lie in the $S^1$ subbundle $\arg_{c_\alpha z^\alpha}^{-1}(0)\to C_\alpha$ should be in correspondence with Lagrangians which lie in the subbundle $\arg_{y_1y_2-\epsilon}^{-1}(0)\subset Y|_{C_\alpha'}$.
In particular monomial admissible Lagrangians of $X$ give us monomial admissible Lagrangians of $\CP^2\setminus E$. 
This allows us to transfer Lagrangians  $L$ in $\Fuk((\CC^*)^2, W_{\Sigma^{int }})$ to Lagrangians $\pi^{-1}(L)\in \Fuk(\CP^2\setminus E, W_E)$.
\begin{remark}
	$\pi^{-1}(L)$ does not arise from a map between the spaces $\CP^2\setminus E$ and $(\CC^*)^2$.
	The symplectic manifold $\CP^2\setminus E$ is constructed from $(\CC^*)^2$ by handle attachment. 
	We keep the notation $\pi^{-1}$ so that it is consistent with the inverse image functor from our earlier discussion on the $B$-model.
\end{remark} 
We observe that the thimbles of the newly introduced nodes (as in \cref{fig:lagrangiansincp2e}) do not arise as lifts of Lagrangians in $(\CC^*)^2$. 
When constructing the Lagrangian thimble, there is a choice of argument in the invariant direction of the node.
 We take the convention that in the local model $Y|_{C_{\alpha}'}$, the argument of the constructed thimble is positive and decreasing to zero along the thimble. 
With this choice of argument an application of the wrapping Hamiltonian will separate the $\tau_i$ and $\pi^{-1}(L)$ so that 
\[\pi^{-1}(L)\cap \theta(\tau_i)=\emptyset.\]
In the monomial admissible category, Lagrangians are infinitesimally wrapped before computing their Floer homology; this infinitesimal wrapping slightly increases the argument of each monomial on the Lagrangian over each area of control. 
A feature of working with monomial admissible category is that monomially admissible Lagrangians which are initially disjoint remain disjoint after the infinitesimal wrapping,
and so we conclude that $\hom(\pi^{-1}(L), \tau_i)=0$. In summary: see \cref{fig:introducingcuts,fig:lagrangiansincp2e}

\begin{assumption}[Monomial Admissible Blow-up II]
	There exists a Lagrangian correspondence between  $((\CC^*)^2, W_\Sigma)$  and $(\CP^2, W_E)$, giving us a functor 
\[
	\pi^{-1}:\Fuk_\Delta((\CC^*)^2, W_\Sigma)\to \Fuk_\Delta(\CP^2\setminus E, W_E).
\]
This functor gives us a semi-orthogonal decompositions of categories:
	\[
		\langle  \pi^{-1} \Fuk_\Delta((\CC^*)^2, W_\Sigma) , \tau_1, \tau_2, \tau_3\rangle.
	\]
We furthermore assume that this is mirror to the decomposition:
\[
	\langle \check \pi^{-1} D^b\Coh(X_{\Sigma^{int }}), \mathcal O_{D_1}, \mathcal O_{D_2}, \mathcal O_{D_3 } \rangle.
\]
	\label{assum:x9111hms}
\end{assumption}
\begin{figure}
	\centering
	\input{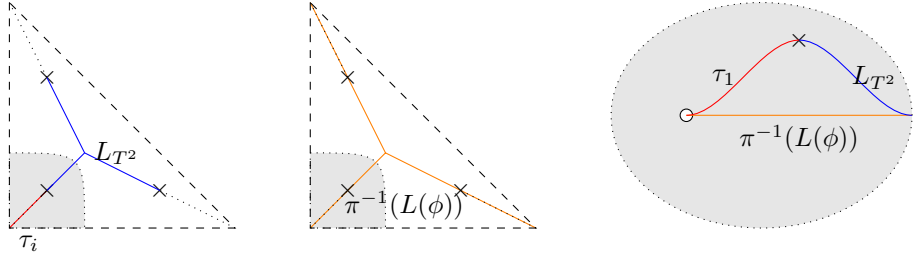}
	\caption[Some Lagrangians in $\CP^2\setminus E$ ]
	{
		The Lagrangians in $\CP^2\setminus E$ relevant to our homological mirror symmetry statement.
	}
	\label{fig:lagrangiansincp2e}
\end{figure}

\begin{remark}
	While to our knowledge this has not been proven for the monomial admissibility condition, this statement is understood by experts in the symplectic Lefschetz fibration admissibility setting \cite{hacking2020homological,auroux2006mirror}. 
	We give a translation of our statement into the Lefschetz viewpoint. 
	Consider the pencil of elliptic curves
	\[	
		p(z_1,z_2,z_3)+t\cdot (z_1z_2z_3).
	\] 
	where $p(z_1z_2z_3)=0$ is homogeneous degree 3 polynomial defining a generic elliptic curve $E$ meeting $z_1z_2z_3=0$ at 9 distinct points. 
	Consider the elliptic fibration $W_{E3}:X_{E3}\to \CP^1$ obtained by blowing up the 9 base points of this elliptic pencil, with exceptional divisors $P_1, \ldots P_9\subset X_{E3}$. 
	Let $z_\infty\in \CP^1$ be a critical value so that $W^{-1}_{E3}(z_\infty)=I_3.$
	Then 
	\[
		(\CC^*)^2 \simeq X_{E3}\setminus( I_3\cup P_1\cup \cdots \cup P_9),\]
	and we may look at the restriction 
	\begin{align*}
		W_{E3}|_{(\CC^*)^2}: (\CC^*)^2\to& \CP^1\setminus \{z_\infty\}= \CC\\
		(z_1, z_2) \mapsto& -\frac{p(z_1, z_2, 1)}{z_1z_2}
	\end{align*}
	By construction, this is a rational function which expands into 9 monomial terms, and has 9 critical points. 
	The nine monomial terms correspond to the 9 directions in the fan drawn in \cref{fig:mirrorsymmetryxint}.
	The Fukaya-Seidel category constructed with $W_{E3}|_{(\CC^*)^2}: (\CC^*)^2\to \CP^1$ is mirror to $X_{\Sigma^{int }}$, where the 9 thimbles drawn from these critical points are mirror to a collection of 9 line bundles generating $D^b\Coh(\check X_{\Sigma^{int }})$.
	These 9 thimbles correspond to 9 tropical Lagrangian sections $\sigma_\phi:Q_{\Sigma^{int }}\to (\CC^*)^2$ in the monomial admissible Fukaya category with fan \cref{fig:mirrorsymmetryxint}.
	The elliptic curve $E$ is a smoothing of the $I_3$ singularity.
	We now consider  
	$X=(\CP^2\setminus E )= (X_{E3}\setminus( E \cup P_1\cup\cdots \cup P_9))$. 
	The restriction 
	\[
		W_{E3}|_{X}: X\to( \CP^1\setminus\{0\})= \CC
	\]
	has  12 critical points, 9 of which may be identified with the critical points from the example before. Conjecturally, this is mirror to $\check X_{9111}$, where the thimbles from the three additional critical points are mirror to the exceptional divisors introduced in the blowup  $\check X_{9111}\to X_{\Sigma^{int }}$.
	In the monomial admissible picture,   the three additional thimbles are matched to the tropical Lagrangian thimbles introduced from the nodes appearing in the toric base diagram $Q_{\CP^2\setminus E}$ drawn in \cref{fig:mirrorsymmetryx9111}
\end{remark}
\begin{table}
	\centering
	\begin{subtable}{.8\linewidth}
		\centering
		\begin{tabular}{c|c}
			A-side & B-side\\ \hline
			$(\CC^*)^2, W_{\Sigma^{int }}$ & $X_{\Sigma^{int }}$\\
			9 Thimbles of $W_{\Sigma^{int }}$ & 9 Line Bundles\\
			$L(\phi_{T^2})$ & Member of 9111-pencil
		\end{tabular}
		\caption{}
	\end{subtable}\\
	\begin{subtable}{.8\linewidth}
		\centering
		\begin{tabular}{c|c}
			A-side & B-side\\ \hline
			$\CP^2\setminus E, W_E$ & $\check X_{9111}$\\
			Thimbles $\tau_i$ &Exceptional Divisors $D_i$\\
			$\pi^{-1}(L(\phi_{T^2}))$ & Total transform of member of 9111 Pencil\\
			$L_{T^2}$ & Fiber of $\check X_{9111}\to \CP^1$. 
		\end{tabular}
		\caption{}
	\end{subtable}
	\caption{A summary of the mirror correspondences that we use for this section.}
\end{table}
\subsubsection{A return to the Lagrangian $L_{T^2}\subset \CP^2\setminus E$.}
We now look at the Lagrangian three punctured torus $L_{{\phi_{T^2}}}\subset (\CC^*)^2=\CP^2\setminus I_3$  described in \cref{exam:ellipticchecker}.
In order to make a homological mirror symmetry statement, we need to use the non-Archimedean mirror $\check X_{9111}^\Lambda$, however the intuition should be independent of the use of Novikov coefficients.

Let $\phi_{T^2}= x_1\oplus x_2\oplus (x_1x_2)^{-1}$ be the tropical polynomial whose critical locus passes through the rays of the nodes added in the modification of $Q_{\Sigma^{int}}$ to $Q_{9111}$.
\begin{theorem}
	There exists a three-ended Lagrangian cobordism with ends
	\[
		(L_{T^2}, \tau_1\cup\tau_2\cup\tau_3)\rightsquigarrow \pi^{-1}(L_{{\phi_{T^2}}}),
	\]
	where $\tau_1\cup \tau_2\cup \tau_3$ is a disconnected Lagrangian in the Fukaya category.
	Provided that \cref{assum:x9111hms} holds and the cobordism is unobstructed, the Lagrangian $L_{T^2}$ is mirror to a divisor Chow-equivalent to a fiber of the elliptic fibration $\check W_{9111}:\check X_{9111}\to \CP^1$.
	\label{thm:fmmirror}
\end{theorem}

\begin{proof}
	We first construct the Lagrangian cobordism. 
	At each of the 3 nodal points in the base of the SYZ fibration $Q_{\CP^2\setminus E}$ the Lagrangian $L_{T^2}$ meets $\tau_i$ at a single intersection point.
	In our local model for the nodal neighborhood, this is the intersection of two Lagrangian thimbles.
	The surgery of those two thimbles is a smooth Lagrangian whose argument in the eigendirection of the node avoids the node.
	This was our local definition for $\pi^{-1}(L(\phi_{T^2}))$ in a neighborhood of the node.

	Recall that in this setting, we have an exact sequence of sheaves
	\begin{equation}
		\left( \bigoplus_{i=1}^3\mathcal O_{D_i}\right)\to \pi^{-1}(\mathcal O_{E_p})\to \mathcal O_{ W_{9111}^{-1}(p) }.
	\end{equation}
	
	In the event that the cobordism constructed above is unobstructed, by \cite{biran2013fukayacategories} we have a similar exact triangle on the $A$-side,
	\[
		\bigsqcup_{i=1}^3 \tau_i \to \pi^{-1}(L(\phi_{T^2}))\to L_{T^2}.
	\]
	\begin{remark}
		It is reasonable to expect that the Lagrangian cobordism in question is unobstructed, as the intersections between the $\tau_i$ and $L_{T^2}$ are all in the same degree, therefore for index reasons we can rule out the existence of holomorphic strips on $L_{T^2}\cup \tau_i$. In complex dimension 2, one can additionally choose an almost complex structure to rule out the existence of Maslov 0 disks with boundary on $L_{T^2}\cup \tau_i$. 
		These are similar to the conditions used to prove unobstructedness of tropical Lagrangian hypersurfaces \cite{hicks2020tropical}.
	\end{remark}
	Under the assumptions of \cite[A.3.2]{hicks2020tropical} on the existence of a restriction morphism for the pearly model of Lagrangian Floer theory of Lagrangian cobordisms , the first and third term in these exact triangles are mirror to each other.
	This identifies the mirror of the middle term in the Chow group, proving the theorem. 
\end{proof}

This mirror symmetry statement ties together several lines of reasoning.
To each fiber $F_q\subset \CP^2\setminus E$ equipped with local system $\nabla$, we can associate a value $OGW(F_q, \nabla)$ which is a weighted count of holomorphic disks with boundary $F_q$ in the compactification $F_q\subset \CP^2$. 
By viewing $\check X_{9111}$ as the moduli space of pairs $(F_q, \nabla)$, we obtain a function 
\[W_{OGW}:\check X_{9111}\setminus I_9\to \CC.\]
This function matches the restriction $\check W_{9111}|_{\check X_{9111}\setminus I_9}$. 
In the previous discussion we conjectured that sheaves supported on $W^{-1}_{OGW}(0)$ are mirror to $L_{T^2}$. 
Recall that $L_{T^2}$ can also be constructed as the dual dimer Lagrangian (\cref{def:dualdimerlagrangian}) to the mutation configuration for the monotone fiber $F_0$.
 In this example, these two constructions of $L_{T^2}$ suggest that the dual dimer Lagrangian for a mutation configuration is mirror to the fiber of the Open Gromov-Witten superpotential. 
 
\printbibliography
\end{document}